\documentclass{amsart}

\usepackage{stmaryrd}

\newtheorem{theorem}{Theorem}[section]

\theoremstyle{definition}

\newtheorem{corollary}[theorem]{Corollary}

\theoremstyle{remark}
\newtheorem*{remark}{Remark}

\numberwithin{equation}{section}

\DeclareMathOperator{\spann}{span}

\newcommand{\B}{{\mathcal B}}
\newcommand{\C}{{\mathcal C}}
\newcommand{\D}{{\mathcal D}}
\newcommand{\N}{{\mathcal N}}
\newcommand{\E}{{\mathcal E}}
\newcommand{\G}{{\mathcal G}}
\newcommand{\M}{{\mathcal M}}
\newcommand{\NN}{\mathbb{N}}
\newcommand{\RR}{\mathbb{R}}

\newcommand{\abs}[1]{{\left|#1\right|}} 
\newcommand{\norm}[1]{{\left\|#1\right\|}}
\newcommand{\Bignorm}[1]{{\Big\|#1\Big\|}}
\newcommand{\biggnorm}[1]{{\bigg\|#1\bigg\|}}

\begin{document}
\numberwithin{equation}{section}

\title[p-summing multiplication operators]{$p$-summing multiplication operators, dyadic Hardy spaces and atomic decomposition}
\author{Paul F.X. M\"{u}ller \and Johanna Penteker }
\address{P.F.X. M\"uller, Institute of Analysis, Johannes Kepler University Linz, Austria, 4040 Linz, Altenberger Strasse 69}
\email{pfxm@bayou.uni-linz.ac.at}
\address{J. Penteker, Institute of Analysis, Johannes Kepler University Linz, Austria, 4040 Linz, Altenberger Strasse 69}
\email{johanna.penteker@jku.at}
\date{\today}

\maketitle

\subsection*{Abstract}
For $u \in H^p$, $0<p\leq 2$, with Haar expansion $u=\sum x_Ih_I$ we constructively determine the Pietsch measure of the $2$-summing multiplication operator
	\[\mathcal{M}_u:\ell^{\infty} \rightarrow H^p, \quad (\varphi_I) \mapsto \sum \varphi_Ix_Ih_I.
\]
Our method yields a constructive proof of Pisier's decomposition of $u \in H^p$
	\[|u|=|x|^{1-\theta}|y|^{\theta}\quad\quad \text{ and }\quad\quad \|x\|_{X_0}^{1-\theta}\|y\|^{\theta}_{H^2}\leq C\|u\|_{H^p},
\] where $X_0$ is Pisier's extrapolation lattice associated to $H^p$ and $H^2$. 
Our construction of the Pietsch measure for the multiplication operator $\mathcal{M}_u$ involves the Haar coefficients of $u$ and its atomic decomposition.

\section{Introduction}
\noindent 
\textbf{The spaces} of this paper are dyadic Hardy spaces $H^p$ with the Haar system $\{h_I\}_{I \in \D}$ (indexed by the dyadic intervals $\D$) as their unconditional basis. For $u=\sum_{I \in \D}x_Ih_I$ we have
	\[\norm{u}_{H^p}=\left( \int_0^1\bigg(\sum_{I \in \D}x_I^21_I(t)\bigg)^{\frac{p}{2}}dt\right)^{\frac{1}{p}}.
\]

\noindent
\textbf{The operators} of this paper are multipliers acting on the Haar system: For $u \in H^p$ with Haar expansion $u=\sum_{I \in \D}x_Ih_I$ the multiplier $\M_u$ is defined by
	\[\M_u:\ell^{\infty}(\D)\rightarrow H^p
\]and
	\[\M_u(\varphi)=\sum_{I \in \D}{\varphi_Ix_Ih_I}. 
\]Clearly $\M_u$ is a bounded operator,
\begin{equation}
\label{eq:introd1}
	\norm{\M_u\varphi}_{H^p} \leq \norm{u}_{H^p}\sup\abs{\varphi_I},
\end{equation}
since $\{h_I\}$ is an unconditional basis of $H^p$. Moreover, for $0<p\leq 2$, $H^p$ is of cotype $2$. Therefore, by the work of Dubinsky, Pe{\l}czy{\'n}ski, Rosenthal \cite{MR0365097} $\M_u$ is $2$-summing. Hence, by the work of Pietsch \cite{MR0216328} $\M_u$ has a Pietsch measure. It is easy to show (Step 2 and 3 in the proof of Theorem \ref{th:vec1}) that the measure has the following form. There exists $\omega=(\omega_I)_{I \in \D}$ with $\omega_I \geq 0 $ and $\sum\omega_I \leq 1$ such that we have a significant strengthening of the basic estimate (\ref{eq:introd1}):
\begin{equation}
	\norm{\M_u\varphi}_{H^p} \leq C\norm{u}_{H^p}\bigg(\sum_{I \in \D}\abs{\varphi_I}^2\omega_I\bigg)^{\frac{1}{2}}.
\end{equation} 
The existence of the weight $\omega=(\omega_I)_{I \in \D}$ is guaranteed by abstract theory (Pietsch measure, Hahn-Banach theorems). 

\subsection*{The main result.} In (Theorem \ref{th:main1}) we give explicit formulas for the weights $\omega=(\omega_I)$, using as input the Haar coefficients $(x_I)_{I \in \D}$ of $u \in H^p$. We obtain several extensions and variants of the formulas referred to above. These include vector-valued dyadic Hardy spaces and Triebel-Lizorkin spaces.

Multiplier operators, as described above, arise with interpolation and extrapolation of Banach lattices. Here we refer to Pisier's proof in \cite{MR557371} of the equation
\begin{equation}
X=(X_0)^{1-\theta}(L^2)^{\theta},
\label{eq:pislattice}
\end{equation}
where $X$ is a $q$-concave and $q'$-convex Banach lattice, $\theta=\frac{2}{q}$ and $X_0$ is Pisier's extrapolation lattice for $X$ and $L^2$ (see Section \ref{sec:application}). The equation (\ref{eq:pislattice}) asserts that for $u \in X$ there is $y \in L^2(\Omega, \Sigma, \mu)$ so that
\begin{equation}
\label{eq:pisierdecomp}
	\left(\abs{u}\abs{y}^{-\theta}\right)^{\frac{1}{1-\theta}} \in X_0.
\end{equation}
In order to obtain $y \in L^2(\Omega, \Sigma, \mu)$ the proof in \cite{MR557371} sets up the following multiplication operator
\begin{equation*}
	\M_u:L^{\infty}(\Omega, \Sigma, \mu) \rightarrow X
\end{equation*} by putting
	\[\M_u(\varphi)(t)=u(t)\varphi(t), \hspace{0.3 cm} t \in \Omega.
\] 
Exploiting the work of Maurey \cite{MR0344931}, Rosenthal \cite{MR0430749} and Pietsch \cite{MR0216328}, Pisier shows in \cite{MR557371} that there exists a density $\omega \in L^1(\Omega, \Sigma, \mu)$ so that 
	\[\norm{\M_u\varphi}_X\leq C\norm{u}_X\left(\int_{\Omega}{\abs{\varphi(t)}^q\omega(t)d\mu(t)}\right)^{\frac{1}{q}}. 
\] To obtain (\ref{eq:pislattice}) and (\ref{eq:pisierdecomp}) Pisier \cite{MR557371} puts finally
\[y(t)=\omega(t)^{\frac{1}{2}}. 
\]
\subsection*{The application.} We constructively determine for the special Banach lattice $X=H^p$ the density $\omega \in \ell^1(\D)$ and therefore have a constructive proof of
	\[ H^p=(X_0)^{1-\theta} (H^2)^{\theta}, 
	\]where $X_0$ is Pisier's extrapolation lattice for $H^p$ and $H^2$, see Theorem \ref{th:app}.\\
	
\noindent
\textbf{The organization} of this paper is the following: Section \ref{sec:prelim} contains the preliminaries. Section \ref{sec:mainth} contains the main result of this paper (Theorem \ref{th:main1}). We constructively determine Pietsch measures for absolutely summing multipliers into $H^p$ spaces. We prove extensions to Triebel-Lizorkin spaces and vector-valued $H^p_X$ spaces.  
In Section \ref{sec:application} we apply our approach to determine constructively the Calder\'on product 
	\[f_p^q=(X_0)^{1-\theta}(f_q^q)^{\theta},
\]where $X_0$ is Pisier's extrapolation lattice for the Triebel-Lizorkin spaces $f_p^q$ and $f_q^q$. 
\subsection*{Acknowledgements}
We would like to thank the referee for a very helpful report improving the presentation of our work.

\section{Preliminaries}
\label{sec:prelim}

\subsection{Banach space preliminaries}

\subsubsection*{Kahane's inequality and Kahane's contraction principle}
See \cite{MR833073}. Let $\{r_n\}_{n \in \NN}$ denote the independent Rademacher system.
For any $0<p<q<\infty$ there is a constant $C_{p,q}$ such that for any Banach space $X$ and for any sequence $(x_n)_{n=1}^{N}$ in $X$ we have
\begin{equation}
\label{eq:Kahane}
\bigg(\int_0^1{\biggnorm{\sum_{n=1}^N{r_n(t)x_n}}_X^qdt}\bigg)^{\frac{1}{q}} \leq C_{p,q} \bigg(\int_0^1{\biggnorm{\sum_{n=1}^N{r_n(t)x_n}}_X^pdt}\bigg)^{\frac{1}{p}}.
\end{equation}
Let $X$ be a Banach space and $(x_n)_{n=1}^{N}$ a sequence in $X$. Then for all sequences $(a_n)_{n=1}^N$ of real numbers and for all $ 1\leq p<\infty$ one has
\begin{equation}
\label{eq:Kahanecon}
	\bigg(\int_0^1{\biggnorm{\sum_{n=1}^N{r_n(t)a_nx_n}}_X^p dt}\bigg)^{\frac{1}{p}} \leq \sup_{1\leq n\leq N}\abs{a_n}\bigg(\int_0^1{\biggnorm{\sum_{n=1}^N{r_n(t)x_n}}_X^p dt}\bigg)^{\frac{1}{p}}.
\end{equation}

\subsubsection*{Cotype of a Banach space}
See e.g. \cite{MR1102015}.
Let $2\leq q\leq \infty$. Let again $\{r_n\}_{n \in \NN}$ denote the independent Rademacher system. A Banach space $X$ is called of \textit{cotype q} if there is a constant $C$ such that for all finite sequences $(x_n)$ in $X$ 
\begin{equation}
\label{eq:cotype}
	\bigg(\sum_{n=1}^N\norm{x_n}^q\bigg)^{\frac{1}{q}}\leq C \bigg(\int_0^1{\biggnorm{\sum_{n=1}^N{r_n(t)x_n}}_X^2dt}\bigg)^{\frac{1}{2}}.
\end{equation}
We denote by $C_q(X)$ the smallest possible constant $C$. 
A Banach space $X$ is of \textit{nontrivial cotype} if it is of cotype $q<\infty$. 

Every $L^p$-space is of cotype $\max(p,2)$. If $X$ is a Banach space of cotype $q$, then $L^r_X$ is of cotype $\max(r,q)$, cf.\cite{MR1102015}.

\subsubsection*{p-summing operators} See e.g. \cite{MR0216328}. 
Let $X,Y$ be Banach spaces and let $1\leq p\leq \infty $. An operator $T\in L(X,Y)$ is called \textit{p-summing} if there is a constant $K$ so that for every choice of an integer $n$ and vectors $\{x_i\}_{i=1}^n$ in $X$, we have
\begin{equation}
\label{eq:summing}
	\left(\sum_{i=1}^n{\norm{Tx_i}^p}\right)^{\frac{1}{p}}\leq K\sup_{\norm{x^*}_{X^*}\leq 1}\left(\sum_{i=1}^n{\abs{x^*(x_i)}}^p\right)^{\frac{1}{p}}.
\end{equation}
The smallest possible constant $K$ is denoted by $\pi_p(T)$. The class of all $p$-summing operators in $L(X,Y)$ is denoted by $\Pi_p(X,Y)$. $\pi_p$ defines a norm on $\Pi_p(X,Y)$ with $\norm{T}\leq \pi_p(T)$ for all $T \in \Pi_p(X,Y)$.

\subsubsection*{Maurey's theorem}
We next recall Maurey's theorem on $p$-summing operators from a $C(K)$-space into a Banach space of nontrivial cotype. We refer to \cite{MR0420319} and the exposition of Maurey's theorem in \cite{MR1342297}.
\begin{theorem}[\cite{MR0420319}]
\label{th:mau}
For all $2 \leq r < \infty$ and $r<s<\infty$ there exists a positive constant $K_{s,r}$ such that for every Banach space $Y$ of cotype $r$ with cotype-r constant $C_r(Y)$ and for every compact Hausdorff space $K$ we have
	\[L(C(K),Y)=\Pi_s(C(K),Y)
\]and for every $T \in L(C(K),Y)$ 
	\[\pi_s(T)\leq  K_{s,r}C_r(Y)\norm{T}.
\]
\end{theorem}
The case when the target space $Y$ is of cotype 2 allows the following strengthening of the conclusion. This is the context of the theorem in \cite{MR0365097}.
\begin{theorem}[\cite{MR0365097}]
\label{th:dpr}
There exists a positive constant $B$ such that for every Banach space $Y$ of cotype $2$ with cotype-$2$ constant $C_2(Y)$ and for every compact Hausdorff space $K$ we have
	\[L(C(K),Y)=\Pi_2(C(K),Y)
\]and for every $T \in  L(C(K),Y)$
	\[\pi_2(T)\leq  BC_2(Y)^2\norm{T}. 
\]
\end{theorem}

\subsubsection*{Banach lattices} For general reference on Banach lattices we refer to \cite{MR540367} and on quasi Banach lattices to \cite{MR752808}, see also \cite{MR557371}.  

Let $(\Omega, \Sigma, \mu)$ be a measure space and $L^0(\Omega, \Sigma, \mu)$ the space of all measurable functions with real values. 
Let $X$ be a (quasi) Banach space, whose elements form a subspace of $L^0(\Omega, \Sigma, \mu)$. We call $X$ a (quasi) \textit{Banach lattice} over the measure space $(\Omega, \Sigma, \mu)$, if for all $f\in X$ and for all $g\in L^0(\Omega,\Sigma, \mu)$ with $ \abs{g} \leq \abs{f}$ holds that $g \in X$ and $\norm{g}_X \leq \norm{f}_X$.

\subsubsection*{q-convexity and q-concavity of Banach lattices} We refer to \cite{MR540367} and for quasi Banach lattices to \cite{MR871851}. 
Let $0 < q \leq \infty$. 
A (quasi) Banach lattice $X$ is called \textit{$q$-convex}, if there exists a constant $M>0$ such that
\begin{align}
\label{eq:convex}
\biggnorm{\left(\sum_{i=1}^n{\abs{x_i}^q}\right)^{\frac{1}{q}}}&\leq M\bigg(\sum_{i=1}^n\norm{x_i}^q\bigg)^{\frac{1}{q}} , 
\end{align}
for every choice of vectors $\{x_i\}_{i=1}^n$ in $X$. The smallest possible constant $M$ is denoted by $M^{(q)}(X)$. 
A (quasi) Banach lattice $X$ is called \textit{$q$-concave}, if there exists a constant $M>0$ such that
\begin{align}
\label{eq:concave}
\bigg(\sum_{i=1}^n\norm{x_i}^q\bigg)^{\frac{1}{q}}&\leq M\biggnorm{\left(\sum_{i=1}^n{\abs{x_i}^q}\right)^{\frac{1}{q}}}, 
\end{align}
for every choice of vectors $\{x_i\}_{i=1}^n$ in $X$. The smallest possible constant $M$ is denoted by $M_{(q)}(X)$.

Every q-concave Banach lattice with $q\geq 2$ is of cotype q, cf.\cite{MR540367,MR1342297}.

\subsubsection*{Multiplication operators}
\label{sec:Mult}
Here we collect crucial information on multiplication operators on C(K) spaces with values in a $q$-concave Banach lattice.

Let $K$ be a compact Hausdorff space and $\mu$ a regular Borel measure on K. Let $1 \leq q < \infty$ and $X$ a $q$-concave Banach lattice over the measure space $(K,\mu)$. Each $x \in X$ induces a bounded multiplication operator 
 	 \[\M_{x}:C(K) \rightarrow X, \varphi \mapsto \varphi \cdot x, \hspace{0.3 cm} \norm{\M_x}=\norm{x}_X.
 \]
Pisier \cite{MR557371} asserts that the multiplication operator is $q$-summing with $\pi_q(\M_x)=M_{(q)}(X)\norm{x}_X$, where $M_{(q)}(X)$ is the $q$-concave constant of $X$.  Explicitly this means that for $\varphi_1,\dots,\varphi_n \in C(K)$
\begin{equation}
\label{eq:pisarg}
\begin{split}
\bigg(\sum_{i=1}^n\norm{\M_x \varphi_i}_X^q\bigg)^{\frac{1}{q}}
&\leq M_{(q)}(X)\norm{x}_X \sup_{\stackrel{\norm{\varphi^*}\leq 1}{\varphi^*\in C(K)^*} }\bigg(\sum_{i=1}^n{\abs{\varphi^*(\varphi_i)}^q}\bigg)^{\frac{1}{q}}. 
\end{split}
\end{equation}

\subsection{Dyadic Hardy spaces}
\subsubsection*{Dyadic intervals}
 An interval $I \subseteq [0,1]$ is called a \textit{dyadic interval}, if there exists $n \in \NN_0$ and $1\leq k \leq 2^n$ such that
	\[I=\left[\frac{k-1}{2^n},\frac{k}{2^n}\right[.
\]

Let $\D=\{I \subseteq [0,1]: I \text{ is dyadic interval}\}$ and $\D_n=\{I \in \D:\abs{I}\geq 2^{-n}\}$. The set of dyadic intervals $\D$ is "nested" in the following sense: 
if $I,J \in \D$ are not disjoint, then either $I \subseteq J$ or $J \subseteq I$.

\subsubsection*{The spaces $\ell^1(\D)$ and $\ell^{\infty}(\D)$}
The space $\ell^1(\D)$ is the space of all summable sequences $s=(s_I)_{I \in \D}$ indexed by the dyadic intervals, i.e. 
	\[\sum_{I \in \D}\abs{s_I}<\infty,
\]
equipped with the norm
	\[\norm{s}_{1}=\sum_{I \in \D}\abs{s_I}.
\]
The space $\ell^{\infty}(\D)$ is the space of all bounded sequences $s=(s_I)_{I \in \D}$ indexed by the dyadic intervals equipped with the norm
	\[\norm{s}_{\infty}=\sup_{I\in \D}\abs{s_I}.
\]

\subsubsection*{Carleson constant}See \cite{MR2157745}.
Let $\E\subseteq \D$ be a non-empty collection of dyadic intervals. Then the Carleson constant of $\E$ is given by
\begin{equation}
\label{eq:carleson}
	\llbracket \mathcal{E}\rrbracket=\sup_{I \in \E}\frac{1}{\abs{I}}\sum_{J \in \E, J \subseteq I}{\abs{J}}. 
\end{equation}

\subsubsection*{Blocks of dyadic intervals}
Let $\mathcal{L}$ be a collection of dyadic intervals. We say that $\mathcal{C}(I)\subseteq \mathcal{L}$ is a block of dyadic intervals in $\mathcal{L}$ if the following conditions hold:

\begin{enumerate}
	\item The collection $\mathcal{C}(I)$ has a unique maximal interval, namely the interval I.
	\item If $J \in \mathcal{C}(I)$ and $K \in \mathcal{L}$, then
	
		\[J \subseteq K \subseteq I \text{ implies } K \in \mathcal{C}(I).
	\]
	\end{enumerate}

\subsubsection*{The Haar system}
We define the $L^{\infty}$- normalised Haar system $\{h_I\}_{I\in \D}$ indexed by dyadic intervals $I$ as follows:
	\[h_I=\begin{cases}
  1  &\text{ on the left half of I},\\
  -1  &\text{ on the right half of I},\\
  0  &\text{ otherwise.}
\end{cases}
\]

\subsubsection*{Dyadic Hardy Spaces $H^p$, $0<p\leq 2$}
\label{sec:hardy}
See \cite{MR2157745, MR2183484}.
Let $(x_I)_{I\in \D}$ be a real sequence. We define $f=(x_I)_{I \in \D}$ to be the real vector indexed by the dyadic intervals. We define the square function of $f$ as follows
\begin{equation}
\label{eq:squaref}
	S(f)(t)=\bigg(\sum_{I \in \D}{x_I^2 1_I(t)}\bigg)^{\frac{1}{2}},
\end{equation}
for $t \in [0,1]$.
The space $H^p$, $0< p \leq 2$, consists of vectors $f=(x_I)_{I \in \D}$ for which
\begin{equation}
\label{eq:hpnorm}
	\norm{f}_{H^p}=\norm{S(f)}_{L^p([0,1])} <\infty.
\end{equation}
For $1\leq p \leq 2$, (\ref{eq:hpnorm}) defines a norm and $H^p$ is a Banach space.  
For $0<p<1$, (\ref{eq:hpnorm}) defines a quasi norm and the resulting Hardy spaces $H^p$ are quasi Banach spaces, cf.\cite{MR1441252}. 
The lattice structure on the $H^p$ spaces is induced by the natural lattice structure on sequence spaces (cf.\cite{MR540367}) and therefore they are (quasi) Banach lattices over the dyadic intervals $\D$ equipped with the counting measure. They are $2$-concave with 2-concavity constant $M_{(2)}(H^p)=1$: let $x^1,\dots,x^n \in H^p$, by Minkowski's inequality for $\frac{2}{p}\geq 1$ we have
\begin{equation}
\begin{split}
\left(\sum_{i=1}^n\norm{x^i}_{H^p}^2\right)^{\frac{1}{2}}&=\left(\sum_{i=1}^n\left(\int_0^1\left(\sum_{I \in \D}\abs{x^i_I}^21_I(t)\right)^{\frac{p}{2}}dt\right)^{\frac{2}{p}}\right)^{\frac{1}{2}}\\
&\leq \left(\int_0^1\left(\sum_{i=1}^n\sum_{I\in \D}\abs{x^i_I}^21_I(t)\right)^{\frac{p}{2}}dt\right)^{\frac{1}{p}}\\
&=\norm{\left(\sum_{i=1}^n{\abs{x^i}^2}\right)^{\frac{1}{2}}}_{H^p}.
\end{split}
\end{equation}
Analogous we get that $H^p$, $0<p\leq 2$, is $p$-convex with $p$-convexity constant $M^{(p)}(H^p)=1$.

For convenience we identify $f=(x_I)_{I \in \D} \in H^p$ with its formal Haar series
	\[f=\sum_{I\in \D}{x_Ih_I}. 
\]

The \textit{Haar support} of $f$ is defined as the following collection of dyadic intervals: $\{J\in \D: x_J \neq 0\}$.
\medskip

\subsubsection*{Discrete Triebel-Lizorkin spaces $f_p^{q}$, $0< p\leq q <\infty$}
For general information on Triebel-Lizorkin spaces we refer to \cite{MR1070037}\footnote{
The Triebel-Lizorkin spaces $f_p^{q}$ are special cases of the discrete Triebel-Lizorkin spaces $f_{p}^{\alpha,q}$, defined in \cite[page 47]{MR1070037}, for the value $\alpha=-\frac{1}{2}$.  }.
Let $(x_I)_{I\in \D}$ be a real sequence. We define $f=(x_I)_{I \in \D}$ to be the real vector indexed by the dyadic intervals. We define the $q$-variation of $f$ as follows
\begin{equation}
	S_q(f)(t)=\bigg(\sum_{I \in \D}{\abs{x_I}^q 1_I(t)}\bigg)^{\frac{1}{q}},
\end{equation}for $t\in [0,1]$. 
The spaces $f_{p}^q$, $0< p\leq q< \infty$, consist of vectors $f=(x_I)_{I \in \D}$ for which
\begin{equation}
\label{eq:tlnorm}
	\norm{f}_{f_{p}^q}=\norm{S_q(f)}_{L^p([0,1])} <\infty.
\end{equation}
For $1\leq p\leq q<\infty$ (\ref{eq:tlnorm}) defines a norm, otherwise it defines only a quasi norm. Therefore, the Triebel-Lizorkin spaces $f_p^q$, $0<p\leq q<\infty$ are (quasi) Banach spaces. 

As in the case of Hardy spaces the lattice structure on the Triebel-Lizorkin spaces is induced by the natural lattice structure on sequence spaces and therefore they are (quasi) Banach lattices over the dyadic intervals equipped with the counting measure. 
\medskip

The Triebel-Lizorkin spaces $f_p^q$ are the $\frac{q}{2}$-convexification of the Hardy space $H^{\frac{2p}{q}}$, where 
$\frac{2p}{q} \in (0,2]$, meaning that $f_p^q$ can be identified with the space of all sequences $u=(x_I)_{I \in \D}$ such that $\abs{u}^{\frac{q}{2}} \in H^{\frac{2p}{q}}$ endowed with the norm $\|\abs{u}^{\frac{q}{2}}\|_{H^{\frac{2p}{q}}}^{\frac{2}{q}}$, cf.\cite{MR540367,MR871851}. The absolute value is defined by $\abs{u}^{\frac{q}{2}}=\left(\abs{x_I}^{\frac{q}{2}}\right)_{I \in \D}$ and it can be identified with its formal Haar series $\abs{u}^{\frac{q}{2}}=\sum_{I \in \D}{\abs{x_I}^{\frac{q}{2}}h_I}$. 

Recall that the spaces $H^p$, $0<p\leq 2$, are $2$-concave and $p$-convex (quasi) Banach lattices. Since $\frac{2p}{q} \in (0,2]$ the theory of convexification (\cite{MR540367,MR871851}) yields that the Triebel-Lizorkin spaces are $q$-concave and $p$-convex Banach lattices. The $q$-concavity and $p$-convexity constants $M_{(q)}(f_p^q)$ and $M^{(p)}(f_p^q)$ are equal to one. 

\subsection{Vector-valued dyadic Hardy Spaces $H^p_X$, $0< p \leq 2$} See \cite{MR2927805}.
Let $X$ be a Banach space and $(x_I)_{I\in \D}$ be a sequence in $X$. We define $f=(x_I)_{I \in \D}$ to be the $X$-valued vector indexed by the dyadic intervals. We define the square function of $f$
as follows:
\begin{equation}
\label{eq:vecsquare}
\mathbb{S}(f)(t)=\lim_{n \rightarrow \infty}\left(\int_0^1{\Bignorm{\sum_{I\in \D_n}{r_I(s)x_Ih_I(t)}}_X^2ds}\right)^{\frac{1}{2}}, \hspace{0.5 cm} t \in[0,1].
\end{equation}
where $\{r_I\}_{I \in \D}$ is an enumeration of the independent Rademacher system. 
Let $0<p\leq 2$. We say $f \in H^p_X$, if 
\begin{equation}
\label{eq:hpnormvec}
	\norm{f}_{H^p_X}=\norm{\mathbb{S}(f)}_{L^p([0,1])}<\infty. 
\end{equation}

We identify $f=(x_I)_{I \in \D}$ with its formal Haar series 
	\[f=\sum_{I \in \D}{x_Ih_I}.
\]
The \textit{Haar support} of $f$ is defined as the following collection of dyadic intervals: $\{J\in \D: x_J \neq 0\}$.

The following theorem states the decomposition of an element in $H^p_X$ into absolutely summing elements with disjoint Haar support and bounded square function. The decomposition is done by a stopping time argument that may be regarded as a constructive algorithm. The decomposition originates in the work of S. Janson and P.W. Jones \cite{MR671315} and is described, for example, in \cite{MR2157745}.

\begin{theorem}[Atomic decomposition]
\label{th:atomdec}
For all $0<p\leq 2$ there exist constants $a_p$, $A_p$ such that for every $u \in H^p_X$ with Haar expansion
	\[u=\sum_{J \in \D}{x_Jh_J}, \hspace{0.5 cm} x_J \in X
\]
there exists an index set $\mathcal{N}\subseteq \NN$ and a sequence $(\mathcal{G}_i)_{i \in \mathcal{N}}$ of blocks of dyadic intervals such that for 
	\[u_i=\sum_{J \in \mathcal{G}_i}{x_Jh_J}, \hspace{0.5 cm} i \in \mathcal{N}
\]the following holds:
\begin{enumerate}
	\item[i)]  $(\mathcal{G}_i)_{i \in \mathcal{N}}$ is a disjoint partition of $\D$.
	\item[ii)] $I_i:=\bigcup_{J \in \mathcal{G}_i}J$ is a dyadic interval and $\E:=\{I_i:i \in \N\}$ satisfies $\llbracket \mathcal{E}\rrbracket\leq 4$.
	\item[iii)] 
	\begin{align}
	\label{eq:atdec}
			a_p \norm{u}_{H^p_X}^p &\leq \sum_{i \in \mathcal{N}}{\norm{u_i}_{H^p_X}^p}\leq \sum_{i \in \mathcal{N}}{\abs{I_i}\norm{\mathbb{S}(u_i)}_{\infty}^p}\leq A_p \norm{u}_{H^p_X}^p.
	\end{align}
\end{enumerate}
The family $(u_i, \mathcal{G}_i, I_i)_{i \in \N}$ is called the atomic decomposition of $u \in H^p_X$. 
\end{theorem}
\begin{remark}
If we set $X=\RR$ in the above theorem, we get the atomic decomposition of $u \in H^p$. Note that in this case $a_p=1$ for all $0<p\leq 2$. 
The scalar valued decomposition procedure, particularly the inequalities (\ref{eq:atdec}) in the case $X=\RR$, can be found in \cite{MR2157745}. 
The right-hand side inequality in (\ref{eq:atdec}) transfers directly from the
scalar valued case (cf.~\cite{MR2157745}) to the vector valued case. For the left-hand side
estimate in (\ref{eq:atdec}) we have to consider two cases. In the case $0<p\leq1$ use the well known triangle inequality
	\[\norm{f+g}_{H^p_X}^p\leq \norm{f}_{H^p_X}^p+\norm{g}_{H^p_X}^p. 
\]
In the case $1<p\leq 2$ there are some differences between the scalar and the vector valued case. In the scalar valued case, the left-hand side inequality follows immediately from the disjoint decomposition of $\D$ into blocks of dyadic intervals. In the vector valued case one can adapt the proof of \cite[Lemma 3.3]{MR2418798}. We include the Appendix (Section \ref{sec:appendix}) in order to give the proof in detail.

However, note that the left-hand side inequality of (\ref{eq:atdec}) depends only on the fact that $(\G_i)_{i\in \N}$ is a sequence of disjoint blocks of dyadic intervals and that the Carleson constant of $\E$ is finite. Therefore, for $\varphi=(\varphi_I)_{I\in \D} \in \ell^{\infty}(\D)$ we have
\begin{align}
\label{eq:atphi}
\biggnorm{\sum_{I \in \D}\varphi_Ix_Ih_I}_{H^p_X}^p\leq \frac{1}{a_p}\sum_{i \in \N}\biggnorm{\sum_{I \in \G_i}\varphi_Ix_Ih_I}_{H^p_X}^p.
\end{align}
\end{remark}

\begin{remark}
Applying the atomic decomposition procedure to the function $\abs{u}^{\frac{q}{2}}=\sum_{I \in \D}\abs{u_I}^{\frac{q}{2}}h_I \in H^{\frac{2p}{q}}$, yields the atomic decomposition of $u=(x_I)_{I \in \D}\in f_p^q$, denoted by $(u_i,I_i,\G_i)$, where $u_i=(x_I)_{I \in \G_i}$ and $I_i$, $\G_i$ are as in Theorem \ref{th:atomdec}.
\end{remark}

\subsection{Maurey's Factorization Theorem}
The following theorem follows directly from Maurey's theorem (Theorem \ref{th:mau}) and from Pietsch's factorization theorem (\cite{MR1144277}).
\begin{theorem}
\label{th:vec1}
Let $X$ be a Banach space of cotype r, $2\leq r < \infty$. 
For all $s>r$ there exists a constant $K_{s,r}>0$ such that the following holds.
For all $f \in H^2_X$ with Haar expansion
	\[f=\sum_{I \in \D}{f_Ih_I}, \hspace{0.5 cm} f_I \in X,
\]there exists a $\mu \in \ell^1(\D)$ (not depending on $s$), with
	\[\mu_I \geq 0, \quad\quad\forall I \in \D
\] and
	\[\sum_{I \in \D}{\mu_I}= 1
\]such that for each $\varphi \in \ell^{\infty}(\D)$
\begin{equation}
\label{eq:stateint2}
\norm{\varphi \cdot f}_{H^2_X}\leq K_{s,r}C_r(X) \norm{f}_{H^2_X}\bigg(\sum_{I \in \D}{\abs{\varphi_I}^s\mu_I}\bigg)^{\frac{1}{s}},
\end{equation}where $C_r(X)$ is the cotype-$r$ constant of $X$ and
	\[\varphi \cdot f=\sum_{I \in \D}{\varphi_If_Ih_I}.
\]
\end{theorem}

\begin{proof}[Proof]
We obtain from Kahane's inequality that $H_{X}^2$ is of cotype $r$. Let $f \in H^2_X$. 
We consider the multiplication operator
\begin{align*}
\M_{f}:\ell^{\infty}(\D) &\rightarrow H^2_X, \varphi \mapsto \varphi \cdot f.
\end{align*}
From Kahane's contraction principle we get $\mathbb{S}(\varphi \cdot f)(t)\leq \sup_{I \in \D}{\abs{\varphi_I}} \mathbb{S}(f)(t).$
Hence, $\norm{\M_{f}}\leq \norm{f}_{H^2_X}$.

\textbf{Step 1}: 
We assume that $f$ has finite Haar support $\D'\subset \D$. Then, applying Maurey's theorem (Theorem \ref{th:mau}) to the multiplication operator $\M_f \in L(\ell^{\infty}(\D'),H^2_{X})$ we obtain that $\M_f$ is $s$-summing for all $s>r$ and
\begin{equation}
\label{eq:ssum}
\pi_{s}(\M_f)\leq  K_{s,r}C_{r}(X)\norm{\M_f}\leq K_{s,r}C_{r}(X)\norm{f}_{H^2_X}.
\end{equation}

Since $\D'$ is finite, we can apply Pietsch's factorization theorem \cite[Theorem III.F.8.]{MR1144277} and get the following. There exists a $\mu \in \ell^1(\D')$ with $\mu_I \geq 0$ for all $I \in \D'$
and $\sum_{I \in \D'}{\mu_I}=1$ such that for each $\varphi \in \ell^{\infty}(\D)$
\begin{equation*}
\norm{\M_f(\varphi)}_{H^2_X}\leq \pi_s(\M_f)\bigg(\sum_{I \in \D'}{\abs{\varphi_I}^s\mu_I}\bigg)^{\frac{1}{s}}.
\end{equation*}

\textbf{Step 2}: 
Let $f \in H^2_X$ with its formal Haar series $f=\sum_{I \in \D}x_Ih_I$, where the partial sums are listed according to the lexicographic order in $\D$. 
The conditional expectation $\mathbb{E}_N$ with respect to the $\sigma$-algebra generated by the set $\{I \in \D:\abs{I}=2^{-N}\}$ is given by
\[\mathbb{E}_N(f)=\sum_{I \in \D_{N-1}}x_Ih_I.\]
Then we have 
\begin{equation}
\label{eq:convergence1}
f=\mathbb{E}_N(f)+f-\mathbb{E}_N(f)
\end{equation}
and for each $\varepsilon >0$ there exists $N(\varepsilon) \in \NN$ such that for all $N \geq N(\varepsilon)$
\begin{equation}
\label{eq:convergence}
\norm{\mathbb{E}_N(f)-f}_{H^2_X}\leq \varepsilon.
\end{equation}
To confirm equation (\ref{eq:convergence}) we exploit the finite cotype of $X$ and invoke Kwapien's theorem (\cite[p.255]{MR1342297}) and Hoffmann-J\o rgensen's theorem (\cite[Theorem 12.3]{MR1342297}). 
We consider $f=(x_I)_{I \in \D}$ with its formal Haar series $f=\sum_{I \in \D}x_Ih_I$ and
define a sequence of independent, symmetric random variables with values in $L^2_X([0,1],dt)$. Let $I \in \D$ and  
\begin{equation}
R_I:[0,1] \rightarrow L^2_X, \quad s \mapsto \Big(t \mapsto x_Ih_I(t)r_I(s)\Big),
\end{equation}where $r_I$ is an enumeration of the Rademacher system. Since $f=\sum_{I\in D}x_I h_I \in H^2_X$ we have that
\[\sup_{N\in \NN}\left(\int_0^1 \int_0^1\biggnorm{\sum_{I \in \D_N}x_Ih_I(t)r_I(s)}_X^2dtds\right)^{\frac{1}{2}}<\infty.\]
Therefore, the partial sums of random variables $\sum_{I \in \D_N}R_I$ are bounded in $L^2_Z([0,1],ds)$, where $Z=L^2_X([0,1],dt)$.
$X$ has finite cotype, hence $L^2_X$ has finite cotype and Kwapien's theorem yields that there exists a limit of the partial sums given by
\[
R(s):=\lim_{N \rightarrow \infty}\sum_{I \in \D_N}R_I(s), \quad \text{a.e.}\]
Hoffmann-J\o rgensen's theorem asserts that the partial sums $\sum_{I \in \D_N}R_I$ converge in $L^2_Z([0,1],ds)$, i.e.
\[\lim_{N \rightarrow \infty}\biggnorm{R-\sum_{I \in \D_N}R_I}_{L^2_Z}=0,\]
and hence (\ref{eq:convergence}) holds. 
Note that we identified $f=(x_I)_{I \in \D}$ with the convergent series $R=\sum_{I \in \D}x_Ih_I\otimes r_I \in L^2_Z$, where $Z=L^2_X$.
\medskip

Iterating (\ref{eq:convergence1}) and (\ref{eq:convergence}) there exists a monotonically increasing sequence $(N_i)_{i\geq 0}$ of natural numbers and a sequence $(f_i)_{i\geq 0}$ in $H^2_X$ given by
\begin{equation}
\label{eq:fi}
f_0=\mathbb{E}_{N_0}(f),\quad
f_i=\mathbb{E}_{N_{i}}(f)-\mathbb{E}_{N_{i-1}}(f),\,  i\geq 1
\end{equation}
such that
\begin{equation}
\label{eq:finorm}
\norm{f_i}_{H^2_X}\leq 4^{-i}\norm{f}_{H^2_X}.
\end{equation}
From the construction of the sequence $(f_i)_{i \geq 0}$ we get that each $f_i$ has finite Haar support $D_i \subset \D$.

\textbf{Step 3}: 
We apply Step 1 to the sequence $(f_i)_{i \geq 0}$. There exists a sequence $(\mu^i)_{i \geq 0}$, $\mu^i \in \ell^1(D_i)$ with $\mu_I^i \geq 0$ for all $I \in D_i$
and $\sum_{I \in D_i}{\mu_I^i}=1$ such that for each $\varphi \in \ell^{\infty}(\D)$
\begin{equation}
\label{eq:desi}
\norm{\M_{f_i}(\varphi)}_{H^2_X}\leq \pi_s(\M_{f_i})\bigg(\sum_{I \in D_i}{\abs{\varphi_I}^s\mu_I^i}\bigg)^{\frac{1}{s}}. 
\end{equation}

Step 2 yields $f=\sum_{i=0}^{\infty}f_i$. Therefore,
\begin{align*}
\norm{\M_f(\varphi)}_{H^2_X}\leq \sum_{i=0}^{\infty}{\norm{\M_{f_i}(\varphi)}_{H^2_X}}\leq\sum_{i=0}^{\infty} \pi_s(\M_{f_i})\bigg(\sum_{I \in D_i}{\abs{\varphi_I}^s\mu_I^i}\bigg)^{\frac{1}{s}}. 
\end{align*}
Using (\ref{eq:ssum}) and (\ref{eq:finorm}) yields
\begin{align*}
\norm{\M_f(\varphi)}_{H^2_X}&\leq K_{s,r}C_{r}(X)\sum_{i=0}^{\infty}\norm{f_i}_{H^2_X}\bigg(\sum_{I \in D_i}{\abs{\varphi_I}^s\mu_I^i}\bigg)^{\frac{1}{s}}\\
&\leq K_{s,r}C_{r}(X)\norm{f}_{H^2_X}\sum_{i=0}^{\infty}4^{-i}\bigg(\sum_{I \in D_i}{\abs{\varphi_I}^s\mu_I^i}\bigg)^{\frac{1}{s}}.
\end{align*}
Let $\frac{1}{s}+\frac{1}{s'}=1$. H\"older's inequality yields
\begin{align*}
\norm{\M_f(\varphi)}_{H^2_X}
&\leq K_{s,r}C_{r}(X)\norm{f}_{H^2_X}\bigg(\sum_{i=0}^{\infty}2^{-is'}\bigg)^{\frac{1}{s'}}\bigg(\sum_{i=0}^{\infty}2^{-is}\sum_{I \in D_i}{\abs{\varphi_I}^s\mu_I^i}\bigg)^{\frac{1}{s}}\\
&\leq \overline{K}_{s,r}C_r(X)\norm{f}_{H^2_X}\bigg(\sum_{i=0}^{\infty}2^{-i}\sum_{I \in D_i}{\abs{\varphi_I}^s\mu_I^i}\bigg)^{\frac{1}{s}}\\
&=\overline{K}_{s,r}C_r(X)\norm{f}_{H^2_X}\bigg(\sum_{I \in \D}\abs{\varphi_I}^s\sum_{i=0}^{\infty}2^{-i}{1_{D_i}(I)\mu_I^i}\bigg)^{\frac{1}{s}}\\
&=\overline{K}_{s,r}C_r(X)\norm{f}_{H^2_X}\bigg(\sum_{I \in \D}\abs{\varphi_I}^s\nu_I\bigg)^{\frac{1}{s}},
\end{align*}where $\nu_I=\sum 2^{-i}{1_{D_i}(I)\mu_I^i}$ satisfies $\nu_I \geq 0$ for all $I \in \D$ and
\[\sum_{I \in \D}{\nu_I}=\sum_{i=0}^{\infty}2^{-i}\sum_{I \in D_i}{\mu_I^i}=\sum_{i=0}^{\infty}2^{-i}=2.\]
\end{proof}

\begin{remark}
The proof above has a direct extension to $H^p_X$, $1\leq p < \infty $, where again the measure $\mu \in \ell^1(\D)$ has its origin in the abstract version of Pietsch's theorem.  
\end{remark}
\begin{remark}
\label{re:q2}
If $X$ is of cotype 2 then we know from Theorem \ref{th:dpr} that the statement of Theorem \ref{th:vec1} is valid for $s\geq 2$. Especially it is valid for $s=2$. 
\end{remark}

\section{The main Theorem}
\label{sec:mainth}
We investigate multiplication operators acting on the Haar system in the Hardy spaces $H^p$, $0<p \leq 2$. We fix $u \in H^p$ with Haar expansion $u=\sum x_Ih_I$ and define 
	$\M_u:\ell^{\infty}(\D) \rightarrow H^p$
by
\begin{equation}
\M_u(\varphi)=\sum_{I \in \D}{\varphi_Ix_Ih_I}.	
\end{equation}
We frequently use the ''lattice convention'' 
\begin{equation}
		\varphi\cdot u=\M_u(\varphi)
\end{equation}
to emphasize that $\varphi \in \ell^{\infty}(\D)$ is acting as a multiplier on $u\in H^p$. Since the Haar basis is 1-unconditional in $H^p$ we have
\begin{equation}
	\norm{\M_u:\ell^{\infty}(\D)\rightarrow H^p}\leq \norm{u}_{H^p}. 
\end{equation}Note that $H^p$, $0<p\leq 2$, is $2$-concave as a sequence space. Hence, our multiplication operator is $2$-summing, see (\ref{eq:pisarg}), and has a Pietsch measure. Precisely, by Step 2 and 3 in the proof of Theorem \ref{th:vec1} this measure is given as follows. There exists a nonnegative sequence $\omega_I$, $I \in \D$ so that $\sum\omega_I\leq 1$ and for all $\varphi \in \ell^{\infty}(\D)$
\begin{equation}
\label{eq:firsttry}
 \norm{\M_u(\varphi)}_{H^p}\leq \norm{u}_{H^p} \left(\sum_{I \in \D}\abs{\varphi_I}^{2}\omega_I\right)^{\frac{1}{2}}.
\end{equation}
Since $\M_u$ is determined by $u \in H^p$, also the Pietsch measure $\left(\omega_I\right)_{I \in \D}$ is given by $u \in H^p$. Note however that the existence of Pietsch measures comes from an application of the Hahn-Banach theorem and therefore Pietsch measures are not given constructively.

\subsection{Construction of Pietsch measures}
The construction of the Pietsch measure is the contribution of our paper: for multipliers ranging in the Hardy spaces $H^p$, we are able to find explicit formulae for the Pietsch measure $(\omega_I)_{I\in \D}$. The input for our construction is the atomic decomposition of $u \in H^p$, Theorem \ref{th:atomdec}. The output is the equation (\ref{eq:omega}) determining $\omega_I$ explicitly. 

Recall that Theorem \ref{th:atomdec} asserts that $u=\sum{x_Ih_I} \in H^p$, $0<p\leq 2$ admits an atomic decomposition, that is a triple $(u_i, \G_i, I_i)_{i \in \N}$ so that
	\[u=\sum_{i \in \N}u_i  \quad \quad \text{ and }\quad\quad u_i=\sum_{I \in \G_i}x_Ih_I 
\]satisfy (\ref{eq:atdec}), explicitly this means 
	\begin{align*}
			\norm{u}_{H^p}^p &\leq \sum_{i \in \mathcal{N}}{\norm{u_i}_{H^p}^p}\leq \sum_{i \in \mathcal{N}}{\abs{I_i}\norm{\mathbb{S}(u_i)}_{\infty}^p}\leq A_p \norm{u}_{H^p}^p.
		\end{align*}

\begin{theorem}
\label{th:main1}
Let $0 <p \leq 2$. Let $u \in H^p$ with Haar expansion
	\[u=\sum_{I \in \D}{x_Ih_I}
\]and atomic decomposition $(u_i, \G_i, I_i)_{i \in \N}$. Then the sequence $\left(\omega_I\right)_{I \in \D}$, defined by
\begin{equation}
\label{eq:omega}
\omega_I=\frac{1}{A_p}\frac{\abs{I_i}^{1-\frac{p}{2}}}{\norm{u_i}_2^{2-p}}\frac{\abs{x_I}^2\abs{I}}{\norm{u}_{H^p}^p}, \hspace{0.5 cm} I \in \G_i,
\end{equation}satisfies

\begin{equation}
	\sum_{I \in \D}{\omega_I}\leq 1
\end{equation}and there exists a constant $C_p>0$ such that for each $\varphi \in \ell^{\infty}(\D)$
\begin{equation}
\label{eq:mainth}
	\norm{\varphi \cdot u}_{H^p}\leq C_p \|u\|_{H^p}\Big(\sum_{I \in \D}{\abs{\varphi_I}^2 \omega_I}\Big)^{\frac{1}{2}}.
\end{equation}
\end{theorem}
\begin{proof}
From
	\[	\norm{u}_{H^p}^p \leq \sum_{i \in \mathcal{N}}{\norm{u_i}_{H^p}^p}
\]we get the estimate
\begin{equation}
\label{eq:pr1}
	\norm{u}_{H^p}^{p}\leq \sum_{i \in \N}{\norm{u_i}_2^p\abs{I_i}^{1-\frac{p}{2}}}.
\end{equation}
We get from 
	\[ \sum_{i \in \mathcal{N}}{\abs{I_i}\norm{S(u_i)}_{\infty}^p}\leq A_p \norm{u}_{H^p}^p
\]that
\begin{equation}
\label{eq:pr2}
	\sum_{i \in \N}{\norm{u_i}_2^p\abs{I_i}^{1-\frac{p}{2}}}\leq A_p\norm{u}_{H^p}^p.
\end{equation}
This follows from
	\[\norm{u_i}_{2}^p\abs{I_i}^{1-\frac{p}{2}}\leq \norm{Su_i}_{\infty}^p\abs{I_i}.
\]
By (\ref{eq:pr1}) and the remark following Theorem \ref{th:atomdec} we get for $\varphi \in \ell^{\infty}(\D)$
\begin{equation}
	\begin{split}
	\biggnorm{\sum_{I \in \D}{\varphi_I x_I h_I}}_{H^p}^p &=\biggnorm{\sum_{i \in \N}\sum_{I \in \G_i}{\varphi_I x_I h_I}}_{H^p}^p\\
	& \leq \sum_{i \in \N}{\biggnorm{\sum_{I \in \G_i}{\varphi_I x_I h_I}}_{2}^p\abs{I_i}^{1-\frac{p}{2}}}\\
	&= \sum_{i \in \N}{\biggnorm{\sum_{I \in \G_i}{\varphi_I \frac{x_I}{\norm{u_i}_2} h_I}}_{2}^p\norm{u_i}_2^p\abs{I_i}^{1-\frac{p}{2}}}.
	\end{split}
\end{equation}
With
\begin{align*}
\biggnorm{\sum_{I \in \G_i}{\varphi_I \frac{x_I}{\norm{u_i}_2} h_I}}_{2}^p &= \Big(\sum_{I \in \G_i}{\varphi_I^2\frac{x_I^2}{\norm{u_i}_2^2}\abs{I}}\Big)^{\frac{p}{2}}
\end{align*}
we get
	\[\biggnorm{\sum_{I \in \D}{\varphi_I x_I h_I}}_{H^p}^p \leq \sum_{i \in \N}{\Big(\sum_{I \in \G_i}{\varphi_I^2\frac{x_I^2}{\norm{u_i}_2^2}\abs{I}}\Big)^{\frac{p}{2}}\norm{u_i}_2^p\abs{I_i}^{1-\frac{p}{2}}}.
\]
Applying H\"older's inequality with $\frac{p}{2}+1-\frac{p}{2}=1$ to 

	\[\sum_{i \in \N}{\Big(\sum_{I \in \G_i}{\varphi_I^2\frac{x_I^2}{\norm{u_i}_2^2}\abs{I}\norm{u_i}_2^p\abs{I_i}^{1-\frac{p}{2}}}\Big)^{\frac{p}{2}}\Big(\norm{u_i}_2^p\abs{I_i}^{1-\frac{p}{2}}\Big)^{1-\frac{p}{2}}}.
\]
we get
	\[\biggnorm{\sum_{I \in \D}{\varphi_I x_I h_I}}_{H^p}^p \leq \Big(\sum_{i \in \N}\sum_{I \in \G_i}{\varphi_I^2\frac{x_I^2}{\norm{u_i}_2^{2-p}}\abs{I}\abs{I_i}^{1-\frac{p}{2}}}\Big)^{\frac{p}{2}}\Big(\sum_{i \in \N}\norm{u_i}_2^p\abs{I_i}^{1-\frac{p}{2}}\Big)^{1-\frac{p}{2}}.
\]
Applying (\ref{eq:pr2}) to the second term on the right-hand side we obtain the estimate
\begin{align*}
\biggnorm{\sum_{I \in \D}{\varphi_I x_I h_I}}_{H^p}^p &\leq A_p^{1-\frac{p}{2}}\norm{u}_{H^p}^{p(1-\frac{p}{2})}\Big(\sum_{i \in \N}\sum_{I \in \G_i}{\varphi_I^2\frac{x_I^2}{\norm{u_i}_2^{2-p}}\abs{I}\abs{I_i}^{1-\frac{p}{2}}}\Big)^{\frac{p}{2}}\\
&= A_p\norm{u}_{H^p}^{p}\Big(\sum_{i \in \N}\sum_{I \in \G_i}{\varphi_I^2\frac{x_I^2}{\norm{u_i}_2^{2-p}\norm{u}_{H^p}^pA_p}\abs{I}\abs{I_i}^{1-\frac{p}{2}}}\Big)^{\frac{p}{2}}.
\end{align*}
Recall
\begin{equation}
\label{eq:pr3}
\norm{u_i}_2^2=\sum_{I \in \G_i}{x_I^2\abs{I}}.	
\end{equation}
By (\ref{eq:pr2}) and (\ref{eq:pr3}) we obtain for the sequence $\left(\omega_I\right)_{I \in \D}$, defined by
	\[\omega_I=\frac{\abs{I_i}^{1-\frac{p}{2}}}{A_p\norm{u}_{H^p}^p}\frac{\abs{I}x_I^2}{\norm{u_i}_2^{2-p}}, \hspace{0.5 cm} I \in \G_i,
\]the following estimate
\begin{align*}
\sum_{I \in \D}{\omega_I}&=\frac{1}{A_p\norm{u}_{H^p}^p}\sum_{i \in \N}\sum_{I \in \G_i}{\frac{\abs{I_i}^{1-\frac{p}{2}}\abs{I}x_I^2}{\norm{u_i}_2^{2-p}}}\\
&=\frac{1}{A_p\norm{u}_{H^p}^p}\sum_{i \in \N}{\abs{I_i}^{1-\frac{p}{2}}\norm{u_i}_2^{p}}\\
&\leq 1.
\end{align*}
\end{proof}

\subsection{Extension to Triebel-Lizorkin spaces} 
We can extend the construction to the Triebel-Lizorkin spaces  $f_p^q$, $0 < p \leq q < \infty$. Recall that $f_p^q$ is the $\frac{q}{2}$-convexification of $H^{\frac{2p}{q}}$, where $\frac{2p}{q} \in (0,2]$, with
\begin{equation}
\label{eq:tr3}
	\norm{u}_{f_p^q}=\norm{\abs{u}^{\frac{q}{2}}}_{H^{\frac{2p}{q}}}^{\frac{2}{q}}.
\end{equation}
Therefore, we have for $\varphi \in \ell^{\infty}(\D)$
\begin{equation}
\label{eq:tr4}
\norm{\varphi\cdot u}_{f_p^q}=\norm{\abs{\varphi}^{\frac{q}{2}}\cdot \abs{u}^{\frac{q}{2}}}_{H^{\frac{2p}{q}}}^{\frac{2}{q}}. 
\end{equation}
Therefore, we get from (\ref{eq:firsttry}), (\ref{eq:tr3}) and (\ref{eq:tr4}) the following statement:

Let $0< p \leq q <\infty$. For all $u \in f_p^q$ with $u=(x_I)_{I \in \D}$ there exists a non negative sequence $\omega_I$, $I \in \D$ so that $\sum{\omega_I}\leq 1$ and for all $\varphi \in \ell^{\infty}$
\begin{equation}
\label{eq:firsttry2} 
\norm{\varphi\cdot u}_{f_p^q}\leq \norm{u}_{f_p^q} \left(\sum_{I \in \D}\abs{\varphi_I}^{q}\omega_I\right)^{\frac{1}{q}}.
\end{equation}where
	\[\varphi \cdot u=(\varphi_Ix_I)_{I \in \D}.
\]

We are now able to give an explicit formula for $\left(\omega_I\right)_{I \in \D}$ using again  (\ref{eq:tr3}) and (\ref{eq:tr4}). 
Let $(u_i,\G_i,I_i)_{i \in \N}$ be the atomic decomposition of $u=(x_I)_{I \in \D} \in f_p^q$. Then $(\abs{u_i}^{\frac{q}{2}},\G_i,I_i)_{i \in \N}$ is the atomic decomposition of $\abs{u}^{\frac{q}{2}} \in H^{\frac{2p}{q}}$. Theorem \ref{th:main1} asserts that there exists a constant $C_{\frac{2p}{q}}$ such that
\begin{equation}
\label{eq:tr2}
	\norm{\abs{\varphi}^{\frac{q}{2}}\cdot \abs{u}^{\frac{q}{2}}}_{H^{\frac{2p}{q}}} \leq C_{\frac{2p}{q}} \Big(\sum_{I \in \D}{\abs{\varphi_I}^{q} \omega_I}\Big)^{\frac{1}{2}}\norm{\abs{u}^{\frac{q}{2}}}_{H^{\frac{2p}{q}}},
\end{equation}
where
\begin{equation*}
\omega_I=\frac{1}{A_{\frac{2p}{q}}}\frac{\abs{I_i}^{1-\frac{p}{q}}}{\norm{\abs{u_i}^{\frac{q}{2}}}_2^{2-\frac{2p}{q}}}\frac{\abs{x_I}^q\abs{I}}{\norm{\abs{u}^{\frac{q}{2}}}_{H^{\frac{2p}{q}}}^{\frac{2p}{q}}}, \hspace{0.5 cm} I \in \G_i. 
\end{equation*}
Summarizing we get the following statement for Triebel-Lizorkin spaces as corollary of Theorem \ref{th:main1}:

\begin{corollary}
\label{co:triebel}
Let $0 <p\leq q < \infty$. Let $u=(x_I)_{I \in \D} \in f_p^q$ with atomic decomposition $(u_i, \G_i, I_i)_{i \in \N}$. Then the sequence $\left(\omega_I\right)_{I \in \D}$, defined by
\begin{equation}
\label{eq:mtriebel1}
\omega_I=\frac{1}{A_{\frac{2p}{q}}}\frac{\abs{I_i}^{1-\frac{p}{q}}}{\norm{u_i}_{f_q^q}^{q-p}}\frac{\abs{x_I}^q\abs{I}}{\norm{u}_{f_p^q}^p}, \hspace{0.5 cm} I \in \G_i,
\end{equation}satisfies

\begin{equation}
\label{eq:mtriebel2}
	\sum_{I \in \D}{\omega_I}\leq 1
\end{equation}and there exists a constant $C_{p,q}>0$ such that for each $\varphi \in \ell^{\infty}(\D)$
\begin{equation}
\|\varphi \cdot u\|_{f_p^q}\leq C_{p,q} \|u\|_{f_p^q}\Big(\sum_{I \in \D}{\abs{\varphi_I}^q \omega_I}\Big)^{\frac{1}{q}}.
\end{equation}
\end{corollary}

\subsection{Extension to vector-valued Hardy spaces}
Let $X$ be a Banach space. Fix a sequence $(x_I)_{I \in \D}$ in $X$, then for $u=\sum_{I \in \D}x_Ih_I \in H^p_X$ we define the multiplication operator 
	\[\M_u:\ell^{\infty}(\D) \rightarrow \overline{\spann}\{x_Ih_I:I \in \D\}\subseteq H^p_X,
\] by
	\[\M_u(\varphi)=\sum_{I \in \D}\varphi_Ix_Ih_I.
\]
By Kahane's contraction principle (\ref{eq:Kahanecon}), for fixed $(x_I)_{I \in \D}$, the sequence $(x_Ih_I)_{I \in \D}$ is an unconditional basic sequence in $H^p_X$, $0<p\leq 2$. 
This remark links the present work on vector-valued $H^p_X$ spaces with the lattices of the previous sections. 

As an application of Theorem \ref{th:vec1} and the atomic decomposition in $H^p_X$ we obtain the following statement. The atomic decomposition works as extrapolation tool, transferring the Pietsch measure of multiplication operators into $H^2_X$ to Pietsch measures for multiplication operators into $H^p_X$, $0<p<2$. 

Here the result is only partially constructive. The formulae for the Pietsch measures for atoms are obtained by applying Theorem \ref{th:vec1} invoking the abstract version of Pietsch's theorem. 

\begin{theorem}
\label{th:main2}
Let $X$ be a Banach space of cotype $r$, $2\leq r < \infty$. 
Let $0 <p \leq 2$ and let $u \in H^p_X$ with Haar expansion
	\[u=\sum_{I \in \D}{x_Ih_I}
\]and atomic decomposition $(u_i, \G_i, I_i)_{i \in \N}$ (see Theorem \ref{th:atomdec}). Then for each $i \in \N$ there exists a $\mu^{(i)} \in \ell^1(\G_i)$ with
\[\mu^{(i)}_I \geq 0, \forall I \in \G_i
\] and
	\[\sum_{I \in \G_i}{\mu^{(i)}_I}=1
\]so that the following holds: The sequence $\left(\omega_I\right)_{I \in \D}$, defined by
\begin{equation}
\omega_I=\frac{\norm{u_i}_{H^2_X}^p\abs{I_i}^{1-\frac{p}{2}}}{A_p\norm{u}_{H^p_X}^p}\mu^{(i)}_I, \hspace{0.5 cm} I \in \G_i, 
\end{equation} satisfies
\begin{equation}
	\sum_{I \in \D}{\omega_I}\leq 1
\end{equation}and for all $s>r$ there exists a constant $K_{s,r,p}>0$ such that for each $\varphi \in \ell^{\infty}(\D)$
\begin{equation}
\label{eq:vecresult}
	\norm{\varphi \cdot u}_{H^p_X}\leq K_{s,r,p}C_r(X)\norm{u}_{H^p_X} \bigg(\sum_{I \in \D}{\abs{\varphi_I}^s \omega_I}\bigg)^{\frac{1}{s}}.
\end{equation}
\end{theorem}

We point out that the exponent $s$ in (\ref{eq:vecresult}) is determined by the cotype of $X$ alone. In particular it is not depending on $0<p\leq 2$. 

\begin{proof} Recall the inequalities in (\ref{eq:atdec}) from the atomic decomposition in Theorem \ref{th:atomdec}. 
We have with
	\[\norm{u}_{H^p_X}^p \leq \frac{1}{a_p}\sum_{i \in \mathcal{N}}{\norm{u_i}_{H^p_X}^p}
\]the following estimate
\begin{equation}
\label{eq:prv1}
	\norm{u}_{H^p_X}^{p}\leq  \frac{1}{a_p}\sum_{i \in \N}{\norm{u_i}_{H^2_X}^p\abs{I_i}^{1-\frac{p}{2}}}.
\end{equation}
We also get from 
	\[ \sum_{i \in \mathcal{N}}{\abs{I_i}\norm{\mathbb{S}(u_i)}_{\infty}^p}\leq A_p \norm{u}_{H^p_X}^p
\]that
\begin{equation}
\label{eq:prv2}
	\sum_{i \in \N}{\norm{u_i}_{H^2_X}^p\abs{I_i}^{1-\frac{p}{2}}}\leq A_p\norm{u}_{H^p_X}^p.
\end{equation}
This follows from
	\[\norm{u_i}_{H^2_X}^p\abs{I_i}^{1-\frac{p}{2}}\leq \norm{\mathbb{S}u_i}_{\infty}^p\abs{I_i}.
\]
By (\ref{eq:prv1}) and the remark following Theorem \ref{th:atomdec} we get for $\varphi \in \ell^{\infty}(\D)$
\begin{equation}
\label{eq:prv3}
	\begin{split}
	\biggnorm{\sum_{I \in \D}{\varphi_I x_I h_I}}_{H^p_X}^p &=\biggnorm{\sum_{i \in \N}\sum_{I \in \G_i}{\varphi_I x_I h_I}}_{H^p_X}^p\\
	& \leq \frac{1}{a_p}\sum_{i \in \N}{\biggnorm{\sum_{I \in \G_i}{\varphi_I x_I h_I}}_{H^2_X}^p\abs{I_i}^{1-\frac{p}{2}}}\\
	&= \frac{1}{a_p}\sum_{i \in \N}{\biggnorm{\sum_{I \in \G_i}{\varphi_I \frac{x_I}{\norm{u_i}_{H^2_X}} h_I}}_{H^2_X}^p\norm{u_i}_{H^2_X}^p\abs{I_i}^{1-\frac{p}{2}}}.
	\end{split}
\end{equation}
We apply Theorem \ref{th:vec1} with the specification $f=u_i$, recall that $u_i=\sum_{I \in \G_i}x_Ih_I$. Therefore, we obtain the following statement: 
For all $s>r$ there exists a constant $K_{s,r}$ such that
	\[\biggnorm{\sum_{I \in \G_i}{\varphi_I \frac{x_I}{\norm{u_i}_{H^2_X}} h_I}}_{H^2_X}^p \leq  K_{s,r}^p C_r(X)^p \bigg(\sum_{I \in \G_i}{\abs{\varphi_I}^s\mu^{(i)}_I}\bigg)^{\frac{p}{s}},
\]where $\mu_i \in \ell^1(\G_i)$ with $\mu^{(i)}_I \geq 0$ for all $I \in \D$ and $\sum_{I \in \G_i}{\mu^{(i)}_I}=1$.
Summing up we get from Maurey's theorem (in particular Theorem \ref{th:vec1}) the following statement

\begin{align}
\label{eq:helpvec1}
\biggnorm{\sum_{I \in \D}{\varphi_I x_I h_I}}_{H^p_X}^p &\leq \frac{K_{s,r}^p C_r(X)^p}{a_p} \sum_{i \in \N}{\bigg(\sum_{I \in \G_i}{\abs{\varphi_I}^s\mu^{(i)}_I}\bigg)^{\frac{p}{s}}\norm{u_i}_{H^2_X}^p\abs{I_i}^{1-\frac{p}{2}}}.
\end{align}

We rewrite the sum on the right-hand side in an appropriate way and apply H\"older's inequality with $\frac{p}{s}+1-\frac{p}{s}=1$:
\begin{equation}
\begin{split}
\label{eq:helpvec2}
\sum_{i \in \N}&{\bigg(\sum_{I \in \G_i}{\abs{\varphi_I}^s\mu^{(i)}_I}\bigg)^{\frac{p}{s}}\norm{u_i}_{H^2_X}^p\abs{I_i}^{1-\frac{p}{2}}}\\
&= \sum_{i \in \N}{\bigg(\sum_{I \in \G_i}{\abs{\varphi_I}^s \norm{u_i}_{H^2_X}^p\abs{I_i}^{1-\frac{p}{2}}\mu^{(i)}_I}\bigg)^{\frac{p}{s}}\bigg(\norm{u_i}_{H^2_X}^p\abs{I_i}^{1-\frac{p}{2}}\bigg)^{1-\frac{p}{s}}}\\
&\leq  {\bigg(\sum_{i \in \N}\sum_{I \in \G_i}{\abs{\varphi_I}^s \norm{u_i}_{H^2_X}^p\abs{I_i}^{1-\frac{p}{2}}\mu^{(i)}_I}\bigg)^{\frac{p}{s}}\bigg(\sum_{i \in \N}\norm{u_i}_{H^2_X}^p\abs{I_i}^{1-\frac{p}{2}}\bigg)^{1-\frac{p}{s}}}.
\end{split}
\end{equation}
By (\ref{eq:prv2}) we get an estimate for the second term on the right-hand side and therefore
\begin{equation}
\begin{split}
\label{eq:helpvec3}
\sum_{i \in \N}&{\bigg(\sum_{I \in \G_i}{\abs{\varphi_I}^s\mu^{(i)}_I}\bigg)^{\frac{p}{s}}\norm{u_i}_{H^2_X}^p\abs{I_i}^{1-\frac{p}{2}}}\\
&\leq  \bigg(\sum_{i \in \N}\sum_{I \in \G_i}{\abs{\varphi_I}^s \norm{u_i}_{H^2_X}^p\abs{I_i}^{1-\frac{p}{2}}\mu^{(i)}_I}\bigg)^{\frac{p}{s}}A_p^{1-\frac{p}{s}} \norm{u}_{H^p_X}^{p(1-\frac{p}{s})}\\
&=A_p \norm{u}_{H^p_X}^{p}\bigg(\sum_{i \in \N}\sum_{I \in \G_i}{\abs{\varphi_I}^s \frac{\norm{u_i}_{H^2_X}^p\abs{I_i}^{1-\frac{p}{2}}}{A_p\norm{u}_{H^p_X}^p}\mu^{(i)}_I}\bigg)^{\frac{p}{s}}.
\end{split}
\end{equation}
Combining (\ref{eq:helpvec1}) and (\ref{eq:helpvec3}) yields
\begin{align*}
\biggnorm{\sum_{I \in \D}{\varphi_I x_I h_I}}_{H^p_X}
&\leq K_{s,r,p} C_r(X)\norm{u}_{H^p_X}\bigg(\sum_{i \in \N}\sum_{I \in \G_i}{\abs{\varphi_I}^s \frac{\norm{u_i}_{H^2_X}^p\abs{I_i}^{1-\frac{p}{2}}}{A_p\norm{u}_{H^p_X}^p}\mu^{(i)}_I}\bigg)^{\frac{1}{s}},
\end{align*}where $K_{s,r,p}$ is dependent on the constant $K_{s,r}$ from Theorem \ref{th:vec1} and the constants $a_p, A_p$ from the atomic decomposition (Theorem \ref{th:atomdec}). 
Now we set
\begin{equation}
	\label{eq:mes}
	\omega_I=\frac{\norm{u_i}_{H^2_X}^p\abs{I_i}^{1-\frac{p}{2}}}{A_p\norm{u}_{H^p_X}^p}\mu^{(i)}_I, \hspace{0.5 cm}I \in \G_i 
\end{equation}and obtain
	\[\biggnorm{\sum_{I \in \D}{\varphi_I x_I h_I}}_{H^p_X} \leq C_r(X)K_{s,r,p}\norm{u}_{H^p_X}\bigg(\sum_{I \in \D}{\abs{\varphi_I}^s \omega_I}\bigg)^{\frac{1}{s}}.
\]
The sequence $(\omega_I)_{I\in \D}$, defined by (\ref{eq:mes}), satisfies
\begin{align*}
\sum_{I \in \D}{\omega_I}=\frac{1}{A_p\norm{u}_{H^p_X}^p}\sum_{i \in \N}\norm{u_i}_{H^2_X}^p\abs{I_i}^{1-\frac{p}{2}}\sum_{I \in \G_i}{\mu^{(i)}_I}.
\end{align*}
Recall that $\mu_i$ is a probability measure on $\G_i$, thus we get from (\ref{eq:prv2})
\begin{align*}
\sum_{I \in \D}{\omega_I}&=\frac{1}{A_p\norm{u}_{H^p_X}^p}\sum_{i \in \N}\norm{u_i}_{H^2_X}^p\abs{I_i}^{1-\frac{p}{2}}\\
&\leq 1.
\end{align*}
\end{proof}
\begin{remark}
If $X$ is of cotype $2$ then the statement of Theorem \ref{th:main2} is valid for all $s\geq 2$ including $s=2$, cf. Theorem \ref{th:dpr} and the remark following Theorem \ref{th:vec1}.
\end{remark}

\section{Application to Pisier's extrapolation lattices}
\label{sec:application}
Fix $0<\theta<1$ and $1<q<\infty$. Define $p,r$ as follows
	\[\frac{1}{p}=1-\theta+\frac{\theta}{q} \quad\quad \text{ and } \quad\quad \frac{1}{r}=\frac{\theta}{q}.
\]
Let $X$ be a lattice over a measure space $(\Omega, \Sigma, \mu)$ which we assume $r$-concave and $p$-convex (with constants one). Let $X_1=L^q(\Omega,\Sigma,\mu)$. The lattice $X_0\subseteq L^0(\Omega,\Sigma,\mu)$ introduced by Pisier in \cite{MR557371} and \cite{MR555306} is defined by putting $x\in X_0$ iff
\begin{equation}
\label{eq:pisierlattice}
\norm{x}_{X_0}=\sup\left\{\norm{\abs{x}^{1-\theta}\abs{y}^{\theta}}_{X}^{\frac{1}{1-\theta}}: \norm{y}_{X_1} \leq 1\right\}< \infty.
\end{equation}
Pisier's theorem, \cite{MR557371}, \cite{MR555306}, asserts that $X_0$ is a Banach lattice and
\begin{equation}
\label{eq:cproduct}
	X=(X_0)^{1-\theta}(X_1)^{\theta}. 
\end{equation}
The lattice $(X_0)^{1-\theta}(X_1)^{\theta}$ is called Calder\'{o}n product of the Banach lattices $X_0$, $X_1$ and is defined as follows (cf.\cite{MR0167830}).  
The lattice $(X_0)^{1-\theta}(X_1)^{\theta}$ is the space of those functions $u \in L^0(\Omega,\Sigma,\mu)$ such that $\abs{u}=\abs{x}^{1-\theta}\abs{y}^{\theta}$ with $x \in X_0$ and $y \in X_1$ equipped with the norm
\begin{equation}
\label{eq:intcal}
\norm{u}_{(X_0)^{1-\theta}(X_1)^{\theta}}=\inf\{\norm{x}_{X_0}^{1-\theta}\norm{y}_{X_1}^{\theta}:\abs{u}=\abs{x}^{1-\theta}\abs{y}^{\theta}\}.
\end{equation}

Specifically, (\ref{eq:cproduct}) states that given $u \in X$ there is $y\in X_1$ and $x\in X_0$ so that 
\begin{equation}
\label{eq:pisierassertion}
	\abs{u}=\abs{x}^{1-\theta}\abs{y}^{\theta} \quad\quad \text{ and } \quad\quad \norm{x}_{X_0}^{1-\theta}\norm{y}_{X_1}^{\theta}\leq C\norm{u}_X. 
\end{equation}
The proof in \cite{MR557371} obtains $y\in X_1$ (and hence $x \in X_0$) by a Hahn-Banach argument, as explained in the third paragraph of the introduction. 
For specific examples of lattices it may however be possible to obtain $y\in X_1$ constructively. 

These considerations were the stimulus for our work on multiplication operators into Hardy spaces $H^p$ (real-valued and vector-valued) and into Triebel-Lizorkin spaces. Recall, the Triebel-Lizorkin spaces $f_p^q$, $1< p \leq q<\infty$, are $p$-convex and $q$-concave Banach lattices over the dyadic intervals equipped with the counting measure. Indeed, $f_p^q$ is $r$-concave for all $r \geq q$. 
M.Frazier and B.Jawerth showed in \cite{MR1070037} that $f_p^q \simeq (f_1^q)^{1-\theta}(f_q^q)^{\theta}$, where  $0<\theta<1$ and $\frac{1}{p}=1-\theta+\frac{\theta}{q}$.
Moreover, Pisier's theorem (\cite[Theorem 2.10, Remark 2.13]{MR557371}) asserts that 
\begin{equation}
f_p^q\simeq(X_0)^{1-\theta}(f_q^q)^{\theta},
\label{eq:intpisier1}
\end{equation}where $X_0$ is the lattice of all elements $f=(x_I)_{I \in \D}$ for which 
\begin{equation}
\norm{f}_{X_0}=\sup\left\{\norm{(\abs{x_I}^{1-\theta}\abs{y_I}^{\theta})_{I\in \D}}_{f_p^q}^{\frac{1}{1-\theta}}: y=(y_I)_{I \in \D} \in f_q^q, \norm{y}_{f_q^q}\leq 1\right\}<\infty.  
\label{eq:intpisier2}
\end{equation}
As we pointed out, the factorization of $u \in f_p^q$ as
\begin{equation}
\label{eq:triebelassertion}
	\abs{u}=\abs{x}^{1-\theta}\abs{y}^{\theta}, \quad \norm{x}_{X_0}^{1-\theta}\norm{y}_{f_q^q}^{\theta}\leq C_{p,q}\norm{u}_{f_p^q}
\end{equation} uses a Pietsch measure for the multiplication operator 
\begin{equation}
\label{eq:mult}
	\M_u:\ell^{\infty}(\D)\rightarrow f_p^q,\quad \varphi \mapsto (\varphi\cdot u).
\end{equation}
Using our explicit formulas for the Pietsch measure of (\ref{eq:mult}), we obtain the decomposition (\ref{eq:triebelassertion}) constructively without invoking Hahn Banach theorems.
\begin{theorem}
\label{th:app}
Let $1< p\leq q < \infty$ and $\theta=\frac{q}{q-1}\frac{p-1}{p}.$
Let $X_0$ be the Banach lattice defined by (\ref{eq:intpisier2}). Let $u=(u_I)_{I \in \D} \in f_p^q$ and let $\omega \in \ell^1(\D)$ be the weight defined by (\ref{eq:mtriebel1}). Then $y \in f_q^q$ and $x\in X_0$ defined by
	\[y=\left(\frac{\omega_I}{\abs{I}}\right)_{I \in \D}^{\frac{1}{q}}\quad\quad \text{ and }\quad\quad x=\left(\abs{u_I}\abs{y_I}^{-\theta}\right)^{\frac{1}{1-\theta}}_{I \in \D}
\]
satisfy
\begin{equation}
	\abs{u}=\abs{x}^{1-\theta}\abs{y}^{\theta} \quad\quad \text{ and }\quad\quad \norm{x}_{X_0}^{1-\theta}\norm{y}_{f_q^q}^{\theta} \leq C_{p,q}\norm{u}_{f_p^q}.
\end{equation}
\end{theorem}
\begin{proof}
We have from Corollary \ref{co:triebel} that
\begin{align*}
\norm{y}_{f_q^q}^q
&=\sum_{I \in \D}{\omega_I}\leq 1. 
\end{align*}
Let $z=(z_I)_{I \in \D} \in f_q^q$ with $\norm{z} \leq 1$.
Since $\frac{p(q-1)}{p-1}\geq q$ we have from H\"{o}lder's inequality
\begin{equation}
\label{eq:help1}
	\left(\sum_{I\in \D}{\abs{y_I}^{-q\theta}\abs{z_I}^{q\theta}\omega_I}\right)^{\frac{1}{q}}\leq \left(\sum_{I\in \D}{\abs{y_I}^{-\frac{p(q-1)}{p-1}\theta}\abs{z_I}^{\frac{p(q-1)}{p-1}\theta}\omega_I}\right)^{\frac{p-1}{p(q-1)}}.
\end{equation}
Using the definition of $y$ and $\theta$ we get 
\begin{equation}
\label{eq:help2}
\begin{split}
\sum_{I\in \D}{\abs{y_I}^{-\frac{p(q-1)}{p-1}\theta}\abs{z_I}^{\frac{p(q-1)}{p-1}\theta}\omega_I}&=\sum_{I\in \D}{\abs{\omega_I}^{-\frac{p(q-1)}{q(p-1)}\theta}\abs{I}^{\frac{p(q-1)}{q(p-1)}\theta}\abs{z_I}^{\frac{p(q-1)}{p-1}\theta}\omega_I}\\
&=\sum_{I\in \D}{\abs{z_I}^{q}\abs{I}}=\norm{z}_{f_q^q}^q\leq 1.
\end{split}
\end{equation}
Therefore, combining (\ref{eq:help1}) and (\ref{eq:help2})
\begin{equation}
\label{eq:help3}
	\left(\sum_{I\in \D}{\abs{y_I}^{-q\theta}\abs{z_I}^{q\theta}\omega_I}\right)^{\frac{1}{q}}\leq 1. 
\end{equation}
We can apply Corollary \ref{co:triebel} to the sequence $\varphi_I=\abs{y_I}^{-\theta}\abs{z_I}^{\theta}$ and get with (\ref{eq:help3})

\begin{equation}
\label{eq:help4}
\begin{split}
\norm{\abs{u}\abs{y}^{-\theta}\abs{z}^{\theta}}_{f_p^q}&\leq C_{p,q}\norm{u}_{f_p^q}\left(\sum_{I \in \D}{\abs{y_I}^{- q\theta}\abs{z_I}^{ q\theta}}\omega_I\right)^{\frac{1}{q}}\\
&\leq C_{p,q}\norm{u}_{f_p^q}.
\end{split}
\end{equation}
Recall that $x=(x_I)_{I \in \D}$ with $x_I=\left(\abs{u_I}\abs{y_I}^{-\theta}\right)^{\frac{1}{1-\theta}}$. Then invoking (\ref{eq:pisierlattice}), the defining equation for the norm in $X_0$, the estimate (\ref{eq:help4}) translates into
\[\norm{x}_{X_0}^{1-\theta}\leq C_{p,q} \norm{u}_{f_p^q}.\] 
As $\norm{y}_{f_q^q}\leq 1$ we have
\begin{align*}
\norm{x}_{X_0}^{1-\theta}\norm{y}_{f_q^q}^{\theta}\leq C_{p,q} \norm{u}_{f_p^q}.
\end{align*}
Since for $I \in \D$
	\[\abs{y_I}^{\theta}\abs{x_I}^{1-\theta}=\left(\frac{\omega_I}{\abs{I}}\right)^{\frac{\theta}{q}}\abs{u_I}\left(\frac{\omega_I}{\abs{I}}\right)^{\frac{-\theta}{q}}=\abs{u_I}
\]we have $\abs{u}=\abs{y}^{\theta}\abs{x}^{1-\theta}$. 
\end{proof}

\begin{remark}
The uniqueness theorem of Cwickel and Nilsson (\cite{MR1996919}) gives the identification of the Banach lattice $X_0$ defined by (\ref{eq:intpisier2}) as $f_1^q$: Let $\frac{1}{p}=1-\theta+\frac{\theta}{q}$ and $f=(x_I)_{I \in \D} \in f_1^q$, then there exists a constant $c$ such that 
\begin{equation}
 c\norm{f}_{f_1^q} \leq \sup\left\{\norm{(\abs{x_I}^{1-\theta}\abs{y_I}^{\theta})_{I\in \D}}_{f_p^q}^{\frac{1}{1-\theta}}: y=(y_I)_{I \in \D} \in f_q^q, \norm{y}_{f_q^q}\leq 1\right\} \leq \norm{f}_{f_1^q}. 
\label{eq:intcwikel}
\end{equation}
Our Theorem \ref{th:app} complements the constructive proofs for (\ref{eq:intcwikel}) given by \cite{MR2183484} and \cite{zbMATH06162617}. The common denominator of Theorem \ref{th:app}, \cite{MR2183484} and \cite{zbMATH06162617} is the use of atomic decomposition as starting point for the proof. 
\end{remark}

\section{Appendix}
\label{sec:appendix}
\subsection*{Proof of the left-hand side of inequality (\ref{eq:atdec}).} The left-hand side inequality of (\ref{eq:atdec}) was stated without proof in \cite{MR2927805}. Since we use this inequality repeatedly in this paper we provide the proof here. It uses the ideas of \cite[Lemma 3.3]{MR2418798}, who in turn exploit the ideas of \cite[Theorem 1]{MR0352938} and \cite[p.336]{MR0425667}.

For the following definitions and statements we refer to \cite{MR2157745} and \cite{MR2418798}. Let $\E \subseteq \D$ be a non-empty collection of dyadic intervals.
We denote by $\E^*$ the set covered by $\E$, i.e. $\E^*=\bigcup_{I \in \E}I.$
In the following we define consecutive generations of $\E$.
We define $\G_0(\E)$ to be the maximal dyadic intervals of $\E$, where maximal refers to inclusion. Note that the maximal intervals of a collection $\E$ are pairwise disjoint intervals and that $\G_0(\E)$ covers the same set as $\E$.  Suppose that we have already defined the generations $\G_0(\E),\dots,\G_{n-1}(\E)$, then we define
	\[\G_{n}(\E)=\G_0(\E\setminus (\G_0(\E)\cup\dots\cup\G_{n-1}(\E))).
\]
Given $I \in \D$, let $I \cap \E=\{J \in \E:J \subseteq I\}$ and put
	\[\G_{\ell}(I,\E)=\G_{\ell}(I \cap \E), \hspace{0.3 cm} \text{ for } \ell \in \NN. 
\] 

We fix pairwise disjoint blocks of dyadic intervals $\{\C(I):I \in \E\}$ so that (\ref{eq:prep1})-(\ref{eq:sigma}) hold:
\begin{align}
	\label{eq:prep1}
	&\llbracket \mathcal{E}\rrbracket =\sup_{I \in \E}\frac{1}{\abs{I}}\sum_{J \in \E, J \subseteq I}{\abs{J}}<\infty,\\
	&\label{eq:prep2}
	\C(I)^*=I,
\end{align}
\begin{equation}
\label{eq:sigma}
\begin{split}
	& \text{The sigma algebra generated by } \{h_J: J\in \C(I)\} \text{ is purely atomic.}\\ 
	& \text{We denote by $\B(I)$ the set of atoms.} 
\end{split}
\end{equation}

For every $I\in \E$ we have
\begin{equation}
\abs{\G_{\ell}^*(I,\E)}\leq 4\cdot 2^{-\frac{2\ell}{4\llbracket \mathcal{E}\rrbracket+1}}\abs{I}. 
\end{equation}
Therefore, by the properties above we get for every $I \in \E$ and every $B \in \B(I)$
\begin{equation}
\label{eq:propei}
	\abs{B\cap \G_{\ell}^*(I,\E)}\leq 4\cdot 2^{-\frac{2\ell}{4\llbracket \mathcal{E}\rrbracket+1}}\abs{B}. 
\end{equation}
Let $x_J \in X$, $J \in \D$ and put for $I \in \E$
	\[u_I=\sum_{J \in \C(I)}{x_Jh_J}.
\] Note that by property (\ref{eq:sigma}) $u_I$ is constant on every atom $B \in \B(I)$. 

\bigskip
Let $1\leq p< \infty$. Then for $u=\sum_{I \in \E} u_I$ we claim
\begin{equation}
	\label{eq:proof}
	\norm{u}_{H^p_X}^p\leq c_p\sum_{I \in \E}{\norm{u_I}_{H^p_X}^p}.
\end{equation}
In order to prove inequality (\ref{eq:proof}) we  define for each $I \in \E$ the set 
	\[A_I=I\setminus \bigcup_{J \in \G_1(I,\E)}J.
\]
Note that by construction $\{A_I:I \in \E\}$ is a collection of pairwise disjoint and measurable sets such that $\bigcup_{I \in \E}A_I=\E^*$, where $\E^*$ is the set covered by $\E$, cf.\cite[Proposition 1]{MR2927805}. 
Therefore,
\begin{equation}
\label{eq:propak2}
\begin{split}
	\Bignorm{\sum_{I \in \E}u_I}_{H^p_X}&=\left(\int_0^1\mathbb S^p\Big(\sum_{I \in \E}u_I\Big)(t)dt\right)^{\frac{1}{p}}=\left(\sum_{K \in \E}\int_{A_K}\mathbb S^p\Big(\sum_{I \in \E}u_I\Big)(t)dt\right)^{\frac{1}{p}}.
	\end{split}
\end{equation}
By the definition of $A_I$ we get
\begin{equation}
\label{eq:defak}
\left(\sum_{K \in \E}\int_{A_K}\mathbb S^p\Big(\sum_{I \in \E}u_I\Big)(t)dt\right)^{\frac{1}{p}}=\bigg(\sum_{K \in \E}\int_{A_K}\mathbb S^p\Big(\sum_{\substack{I \in \E \\ I \supseteq K}}u_I\Big)(t)dt\bigg)^{\frac{1}{p}}.
\end{equation}
We know that $\G_0(K,\E)=K$. There exists a shortest dyadic interval $\G_{-1}(K,\E) \in \E$ such that $K$ is strictly contained in $\G_{-1}(K,\E)$. Then there exists a shortest dyadic interval $\G_{-2}(K,\E) \in \E$ such that $\G_{-1}(K,\E)$ is strictly contained in $\G_{-2}(K,\E)$. We continue this pattern $n(K)$ steps until $\G_{-n(K)}(K,\E)$ is a maximal interval in $\E$ and therefore not contained in any interval in $\E$. We have
	\[K=\G_{0}(K,\E)\subset\G_{-1}(K,\E)\subset \G_{-2}(K,\E)\subset \dots \subset \G_{-n(K)}(K,\E).
\]
Thus 
\begin{equation}
\label{eq:gen}
	\sum_{I \in \E,\, I \supseteq K}u_I=\sum_{\ell=0}^{n(K)}u_{\G_{-\ell}(K,\E)}.
\end{equation}
Summarizing the equations (\ref{eq:propak2}), (\ref{eq:defak}) and (\ref{eq:gen}) we have
\begin{equation}
\begin{split}
\Bignorm{\sum_{I \in \E}u_I}_{H^p_X}&= \left(\sum_{K \in \E}\int_{A_K}\mathbb S^p\left(\sum_{\ell=0}^{\infty}1_{[0,n(K)]}(\ell)u_{\G_{-\ell}(K,\E)}\right)(t)dt\right)^{\frac{1}{p}}.
\label{eq:sum1}
\end{split}
\end{equation}
Put $\C_{\ell,K}=\C(\G_{-\ell}(K,\E))$, then $u_{\G_{-\ell}(K,\E)}=\sum_{J \in \C_{\ell,K}}x_Jh_J$ and we have
\begin{equation}
\begin{aligned}
\Bignorm{\sum_{I \in \E}u_I}_{H^p_X}&=\bigg(\sum_{K \in \E}\int_{A_K} \bigg(\int_{0}^1{\biggnorm{\sum_{\ell=0}^{\infty}1_{[0,n(K)]}(\ell)\sum_{J }x_Jh_J(t)r_J(s)}_X^2 ds}\bigg)^{\frac{p}{2}}dt\bigg)^{\frac{1}{p}}.
\end{aligned}
\label{eq:sum2}
\end{equation} 
We define the sequence $(a_{\ell})_{\ell=0}^{\infty}$ of elements in $L^p_{L^2_X}(A_K \times \E, dtdz)$, where $dz$ is the counting measure on $\E$, as follows: $a_{\ell}(t,K)=1_{[0,n(K)]}(\ell)\sum_{J}x_Jh_J(t)r_J$. By the triangle inequality we get
\begin{equation}
\begin{split}
&\left(\sum_{K \in \E}\int_{A_K}{\biggnorm{\sum_{\ell=0}^{\infty}a_{\ell}(t,K)}_{L^2_X}^p}dt\right)^{\frac{1}{p}}\leq \sum_{\ell=0}^{\infty}\left(\sum_{K \in \E}\int_{A_K} {\Bignorm{a_{\ell}(t,K)}_{L^2_X}^p }dt\right)^{\frac{1}{p}}\\
&\quad\quad=\sum_{\ell=0}^{\infty}\left(\sum_{K \in \E}1_{[0,n(K)]}(\ell)\int_{A_K} \bigg(\int_{0}^1{\biggnorm{\sum_{J \in \C_{\ell,K}}x_Jh_J(t)r_J(s)}_X^2 ds}\bigg)^{\frac{p}{2}}dt\right)^{\frac{1}{p}}\\
&\quad\quad=\sum_{\ell=0}^{\infty}\left(\sum_{K \in \E}1_{[0,n(K)]}(\ell)\int_{A_K} \mathbb{S}^p\left(u_{\G_{-\ell}(K,\E)}\right)(t)dt\right)^{\frac{1}{p}}.
\label{eq:triangle}
\end{split}
\end{equation}
Combining inequality (\ref{eq:sum2}) and (\ref{eq:triangle}) we obtain
\begin{equation}
	\Bignorm{\sum_{I \in \E}u_I}_{H^p_X}\leq \sum_{\ell=0}^{\infty}\left(\sum_{K \in \E}1_{[0,n(K)]}(\ell)\int_{A_K} \mathbb{S}^p\left(u_{\G_{-\ell}(K,\E)}\right)(t)dt\right)^{\frac{1}{p}}.
	\label{eq:sum3}
\end{equation}
By the definition of $\G_{-\ell}(K,\E)$ we know that for each $\ell \in [0,n(K)]$ there exists a unique $I \in \E$ such that $\G_{-\ell}(K,\E)=I$. Therefore,
\begin{equation}
	\label{eq:rename1}
\sum_{K \in \E}1_{[0,n(K)]}(\ell)\int_{A_K} \mathbb{S}^p\left(u_{\G_{-\ell}(K,\E)}\right)(t)dt
		=\sum_{K \in \E}\sum_{\substack{I \in \E\\ \G_{-\ell}(K,\E)=I}}\int_{A_K} \mathbb{S}^p\left(u_{I}\right)(t)dt.
\end{equation}
Since the set $\{I,K \in \E:I=\G_{-\ell}(K,\E)\}$ contains the same elements as $\{I,K \in \E:K \in \G_{\ell}(I,\E)\}$ we can rewrite the sums in (\ref{eq:rename1}) and obtain with inequality (\ref{eq:sum3}):
\begin{equation}
	\label{eq:rename2}
		\Bignorm{\sum_{I \in \E}u_I}_{H^p_X}\leq \sum_{\ell=0}^{\infty}\bigg(\sum_{I \in \E}\sum_{ K\in \G_{\ell}(I,\E)}\int_{K} \mathbb{S}^p(u_{I})(t)dt\bigg)^{\frac{1}{p}}.
\end{equation}
We used above that $A_K \subseteq K$. 
For $\ell=0$ we obtain for the right-hand side in (\ref{eq:rename2}):
\begin{equation}
\label{eq:l0}
	\bigg(\sum_{I \in \E}\int_{I} \mathbb{S}^p(u_{I})(t)dt\bigg)^{\frac{1}{p}}=\bigg(\sum_{I \in \E}\norm{u_I}_{H^p_X}^p\bigg)^{\frac{1}{p}}.
\end{equation}
For $\ell\geq1$ we can rewrite the right-hand side in (\ref{eq:rename2}) as follows:
\begin{equation}
\label{eq:l2}
\sum_{K \in \G_{\ell}(I,\E)}\int_{K} \mathbb{S}^p(u_{I})(t)dt=\sum_{B\in \B(I)}\sum_{\substack{K \in \G_{\ell}(I,\E)\\ K \subseteq B}}\int_{K} \mathbb{S}^p(u_{I})(B)dt.
\end{equation}
 
By (\ref{eq:sigma}) $u_I$ is constant on each atom $B \in \B(I)$ and the term $\mathbb{S}^p(u_{I})(B)$ is well-defined. Therefore, we get
\begin{equation}
\label{eq:l22}
	\sum_{B\in \B(I)}\sum_{\substack{K \in \G_{\ell}(I,\E)\\ K \subseteq B}}\int_{K} \mathbb{S}^p(u_{I})(B)dt=\sum_{B\in \B(I)} \mathbb{S}^p(u_{I})(B)\sum_{\substack{K \in \G_{\ell}(I,\E)\\ K \subseteq B}}\abs{K}.
\end{equation}
By inequality (\ref{eq:propei}) we have
	\[\sum_{\substack{K \in \G_{\ell}(I,\E)\\ K \subseteq B}}\abs{K}=\abs{B\cap\G_\ell^*(I,\E)} \leq 4\cdot 2^{-\frac{2\ell}{4\llbracket \mathcal{E}\rrbracket+1}}\abs{B}.
\]
We get the following estimate for the right-hand side in (\ref{eq:l22})
\begin{equation}
\begin{split}
\label{eq:l222}
\sum_{B\in \B(I)} \mathbb{S}^p(u_{I})(B)\sum_{\substack{K \in \G_{\ell}(I,\E)\\ K \subseteq B}}\abs{K} &\leq 4\cdot 2^{-\frac{2\ell}{4\llbracket \mathcal{E}\rrbracket+1}}\sum_{B\in \B(I)} \mathbb{S}^p(u_{I})(B) \abs{B}\\
&= 4\cdot 2^{-\frac{2\ell}{4\llbracket \mathcal{E}\rrbracket+1}}\int_{0}^1{\mathbb{S}^p(u_{I})(t)dt}\\
&= 4\cdot 2^{-\frac{2\ell}{4\llbracket \mathcal{E}\rrbracket+1}}\norm{u_I}_{H_X^p}^p.
\end{split}
\end{equation}
Combining inequalities (\ref{eq:rename2}), (\ref{eq:l0}), (\ref{eq:l2}), (\ref{eq:l22}) and (\ref{eq:l222}) we obtain
\begin{equation}
	\Bignorm{\sum_{I \in \E}u_I}_{H^p_X}\leq \Big(1+4^{\frac{1}{p}}\sum_{l=1}^{\infty}2^{-\frac{2\ell}{p(4\llbracket \mathcal{E}\rrbracket+1)}}\Big)\bigg(\sum_{I \in \E}\norm{u_I}_{H_X^p}^p\bigg)^{\frac{1}{p}}.
\end{equation}

\subsection*{Acknowledgements}
This paper is part of the second named author's PhD thesis written at the Department of Analysis, J. Kepler University Linz. This research has been supported by the Austrian Science foundation (FWF) Pr.Nr.P23987 and Pr.Nr.P22549 and by the NSF sponsored Workshop in Analysis and Probability at Texas A\&M University, 2013. 
\bibliographystyle{alpha}
\bibliography{bib}

\begin{thebibliography}{Mau74b}

\bibitem[Bow13]{zbMATH06162617}
M.~Bownik.
\newblock {Extrapolation of discrete Triebel-Lizorkin spaces.}
\newblock {\em {Math. Nachr.}}, 286(5-6):492--502, 2013.

\bibitem[Cal64]{MR0167830}
A.-P. Calder{\'o}n.
\newblock Intermediate spaces and interpolation, the complex method.
\newblock {\em Studia Math.}, 24:113--190, 1964.

\bibitem[CNS03]{MR1996919}
M.~Cwikel, P.~G. Nilsson, and G.~Schechtman.
\newblock Interpolation of weighted {B}anach lattices. {A} characterization of
  relatively decomposable {B}anach lattices.
\newblock {\em Mem. Amer. Math. Soc.}, 165(787):vi+127, 2003.

\bibitem[CT86]{MR871851}
B.~Cuartero and M.~A. Triana.
\newblock {$(p,q)$}-convexity in quasi-{B}anach lattices and applications.
\newblock {\em Studia Math.}, 84(2):113--124, 1986.

\bibitem[DJT95]{MR1342297}
J.~Diestel, H.~Jarchow, and A.~Tonge.
\newblock {\em Absolutely summing operators}, volume~43 of {\em Cambridge
  Studies in Advanced Mathematics}.
\newblock Cambridge University Press, Cambridge, 1995.

\bibitem[DPR72]{MR0365097}
E.~Dubinsky, A.~Pe{\l}czy{\'n}ski, and H.~P. Rosenthal.
\newblock On {B}anach spaces {$X$} for which {$\Pi _{2}({\mathcal L}_{\infty
  },\,X)=B({\mathcal L}_{\infty },\,X)$}.
\newblock {\em Studia Math.}, 44:617--648, 1972.
\newblock Collection of articles honoring the completion by Antoni Zygmund of
  50 years of scientific activity, VI.

\bibitem[FJ90]{MR1070037}
M.~Frazier and B.~Jawerth.
\newblock A discrete transform and decompositions of distribution spaces.
\newblock {\em J. Funct. Anal.}, 93(1):34--170, 1990.

\bibitem[GM08]{MR2418798}
S.~Geiss and P.~F.~X. M{\"u}ller.
\newblock Haar type and {C}arleson constants.
\newblock {\em Bull. Lond. Math. Soc.}, 40(3):432--438, 2008.

\bibitem[GMP05]{MR2183484}
S.~Geiss, P.~F.~X. M{\"u}ller, and V.~Pillwein.
\newblock A remark on extrapolation of rearrangement operators on dyadic
  {$H^s$}, {$0<s\leq 1$}.
\newblock {\em Studia Math.}, 171(2):197--205, 2005.

\bibitem[JJ82]{MR671315}
S.~Janson and P.~W. Jones.
\newblock Interpolation between {$H^{p}$} spaces: the complex method.
\newblock {\em J. Funct. Anal.}, 48(1):58--80, 1982.

\bibitem[JO74]{MR0352938}
W.~B. Johnson and E.~Odell.
\newblock Subspaces of {$L_{p}$} which embed into {$l_{p}$}.
\newblock {\em Compositio Math.}, 28:37--49, 1974.

\bibitem[Joh76]{MR0425667}
W.~B. Johnson.
\newblock Operators into {$L_{p}$} which factor through {$1_{p}$}.
\newblock {\em J. London Math. Soc. (2)}, 14(2):333--339, 1976.

\bibitem[Kah85]{MR833073}
J.~P. Kahane.
\newblock {\em Some random series of functions}, volume~5 of {\em Cambridge
  Studies in Advanced Mathematics}.
\newblock Cambridge University Press, Cambridge, second edition, 1985.

\bibitem[Kal84]{MR752808}
N.~J. Kalton.
\newblock Convexity conditions for nonlocally convex lattices.
\newblock {\em Glasgow Math. J.}, 25(2):141--152, 1984.

\bibitem[LT79]{MR540367}
J.~Lindenstrauss and L.~Tzafriri.
\newblock {\em Classical {B}anach spaces. {II}}, volume~97 of {\em Ergebnisse
  der Mathematik und ihrer Grenzgebiete [Results in Mathematics and Related
  Areas]}.
\newblock Springer-Verlag, Berlin, 1979.
\newblock Function spaces.

\bibitem[LT91]{MR1102015}
M.~Ledoux and M.~Talagrand.
\newblock {\em Probability in {B}anach spaces}, volume~23 of {\em Ergebnisse
  der Mathematik und ihrer Grenzgebiete (3) [Results in Mathematics and Related
  Areas (3)]}.
\newblock Springer-Verlag, Berlin, 1991.
\newblock Isoperimetry and processes.

\bibitem[Mau74a]{MR0344931}
B.~Maurey.
\newblock {\em Th\'eor\`emes de factorisation pour les op\'erateurs lin\'eaires
  \`a valeurs dans les espaces {$L^{p}$}}.
\newblock Soci\'et\'e Math\'ematique de France, Paris, 1974.
\newblock With an English summary, Ast{\'e}risque, No. 11.

\bibitem[Mau74b]{MR0420319}
B.~Maurey.
\newblock Une nouvelle caract\'erisation des applications {$(p,q)$}-sommantes.
\newblock In {\em S\'eminaire {M}aurey-{S}chwartz 1973--1974: {E}spaces
  {L{$\sup{p}$}}, applications radonifiantes et g\'eom\'etrie des espaces de
  {B}anach, {E}xp. {N}o. 12}, pages 16 pp. (errata, p. E.1). Centre de Math.,
  \'Ecole Polytech., Paris, 1974.

\bibitem[M{\"u}l05]{MR2157745}
P.~F.~X. M{\"u}ller.
\newblock {\em Isomorphisms between {$H^1$} spaces}, volume~66 of {\em Instytut
  Matematyczny Polskiej Akademii Nauk. Monografie Matematyczne (New Series)
  [Mathematics Institute of the Polish Academy of Sciences. Mathematical
  Monographs (New Series)]}.
\newblock Birkh\"auser Verlag, Basel, 2005.

\bibitem[M{\"u}l12]{MR2927805}
P.~F.~X. M{\"u}ller.
\newblock Extrapolation of vector-valued rearrangement operators {II}.
\newblock {\em J. Lond. Math. Soc. (2)}, 85(3):722--736, 2012.

\bibitem[Pie67]{MR0216328}
A.~Pietsch.
\newblock Absolut {$p$}-summierende {A}bbildungen in normierten {R}\"aumen.
\newblock {\em Studia Math.}, 28:333--353, 1966/1967.

\bibitem[Pis79a]{MR557371}
G.~Pisier.
\newblock La m\'ethode d'interpolation complexe: applications aux treillis de
  {B}anach.
\newblock In {\em S\'eminaire d'{A}nalyse {F}onctionnelle (1978--1979)}, pages
  Exp. No. 17, 18. \'Ecole Polytech., Palaiseau, 1979.

\bibitem[Pis79b]{MR555306}
G.~Pisier.
\newblock Some applications of the complex interpolation method to {B}anach
  lattices.
\newblock {\em J. Analyse Math.}, 35:264--281, 1979.

\bibitem[Ros76]{MR0430749}
H.~P. Rosenthal.
\newblock Some applications of {$p$}-summing operators to {B}anach space
  theory.
\newblock {\em Studia Math.}, 58(1):21--43, 1976.

\bibitem[Woj91]{MR1144277}
P.~Wojtaszczyk.
\newblock {\em Banach spaces for analysts}, volume~25 of {\em Cambridge Studies
  in Advanced Mathematics}.
\newblock Cambridge University Press, Cambridge, 1991.

\bibitem[Woj97]{MR1441252}
P.~Wojtaszczyk.
\newblock Uniqueness of unconditional bases in quasi-{B}anach spaces with
  applications to {H}ardy spaces. {II}.
\newblock {\em Israel J. Math.}, 97:253--280, 1997.

\end{thebibliography}
\end{document}